\documentclass[a4paper]{article}

\usepackage{a4wide}
\usepackage{latexsym}
\usepackage[UKenglish]{babel}
\usepackage{graphicx}
\usepackage{subcaption}
\usepackage{algorithm}
\usepackage{algpseudocode}
\usepackage{epstopdf}
\usepackage[colorlinks=true,linkcolor=blue,citecolor=blue]{hyperref}
\usepackage{xcolor}
\usepackage{booktabs}
\usepackage{amsfonts,amsmath,amsthm,mathtools,bbm,cool}

\newtheorem{theorem}{Theorem}[section]

\theoremstyle{remark}

\newcommand{\be}{\begin{equation}}
\newcommand{\ee}{\end{equation}}

\usepackage{color}

\allowdisplaybreaks

\begin{document}
\title{Kinetic and mean-field modeling of muscular dystrophies}

\author{
		Tommaso Lorenzi\\
		{\small	Department of Mathematical Sciences ``G. L. Lagrange"} \\
		{\small Politecnico di Torino, Italy} \\
		{\small\tt tommaso.lorenzi@polito.it} 
		\\
		Horacio Tettamanti\\
		{\small	Department of Mathematics ``F. Casorati''} \\
		{\small University of Pavia, Italy} \\
		{\small\tt horacio.tettamanti01@universitadipavia.it }
		\\
		 Mattia Zanella \\
		{\small	Department of Mathematics ``F. Casorati''} \\
		{\small University of Pavia, Italy} \\
		{\small\tt mattia.zanella@unipv.it} 
		}
		\date{}

\maketitle

\begin{abstract} We present a new class of models for assessing the cell dynamics characterising muscular dystrophies. The proposed approach comprises a system of integro-differential equations for the statistical distributions, over a large patient cohort, of the densities of muscle fibers and immune cells implicated in muscle inflammation, degeneration, and regeneration, which underpin disease development. Considering an appropriately scaled version of this model, we formally derive, as the corresponding mean-field limit, a system of Fokker-Planck equations, from which we subsequently derive, as a  macroscopic model counterpart, a system of differential equations for the mean densities of muscle and immune cells in the cohort of patients and the related variances. Then, we study long-time asymptotics for the mean-field model by determining the quasi-equilibrium cell distribution functions, which are in the form of probability density functions of inverse Gamma distributions, and proving the long-time convergence to such quasi-equilibrium distributions. The analytical results obtained are illustrated by means of a sample of results of numerical simulations. The modeling approach presented here has the potential to offer new insights into the balance between degeneration and regeneration mechanisms in the progression of muscular dystrophies, and provides a basis for future extensions, including the modeling of therapeutic interventions.
\end{abstract}

\section{Introduction}

Muscular dystrophies (MD) are commonly defined as congenital myogenic disorders charaterized by a progressive damage of the muscular tissue with variable distribution and severity, see e.g. \cite{EMQ} and references therein. Among the most common groups of dystrophies, the Duchenne and Becker MDs are the most common and serious forms of this disease. In Duchenne muscular dystrophy (DMD), onset typically occurs in early childhood with generalized weakness of the limbs. This leads to a progressive loss of ambulation from early adolescence onward, accompanied by a decline in respiratory capacity and progressive cardiac involvement, ultimately resulting in death \cite{Emery}. 

At the cellular level, in healthy muscle regeneration, macrophages clear cellular debris, resolve inflammation, and secrete factors that promote tissue repair. In dystrophic conditions, such as DMD, repeated muscle fiber damage triggers the infiltration of immune cells, including cytotoxic T lymphocytes, neutrophils, and macrophages \cite{TVMTF}. These cells release pro-inflammatory mediators that disrupt the normal regenerative process, promote chronic inflammation, and exacerbate muscle degeneration, fat replacement, reactive myofibrosis, and chronic inflammation \cite{DSO}. Over time, persistent immune activation contributes to fibrosis and can extend to cardiac muscle, leading to cardiomyopathy and cardiac dysfunction.

In fact, while mechanical injury and membrane defects are recognized as key contributors to the pathology of dystrophic diseases, they do not fully account for the onset and progression of MDs. Increasing evidence points to inflammation-driven mechanisms as early triggers of lesion formation in dystrophin-deficient muscle -- see e.g. \cite{TVMTF} and references therein. The immune response is believed to play a significant role in driving muscle damage. Supporting this view, histological analyses of dystrophic muscle frequently reveal extensive infiltration by cytotoxic T lymphocytes, macrophages, and other immune cells. In this respect, the pathology progresses through a cascade of events that can be broadly outlined as follows: macrophages are the first to infiltrate the damaged muscle tissue, where they begin clearing cellular debris. During this process, they attract cytotoxic T lymphocytes, which begin attacking not only compromised cells but also healthy muscle cells, thereby exacerbating the damage.

Only few mathematical models are currently available to describe cell dynamics in MD progression and to capture the interplay between inflammation, degeneration, and regeneration -- see \cite{HCG} for a review on this topic. Among them we mention the compartmental models formulated as ordinary differential equations employed in \cite{DAC,HG,JCEGL}. Developing mathematical models capable of integrating the roles of different immune cell populations and their mediators among large cohorts of MD patients remains an open challenge. Addressing this challenge could provide insight into disease progression and could support the design of targeted therapeutic strategies. 

An increasing number of real world phenomena have been fruitfully described by kinetic and mean-field models. Particular attention has been paid in the past decade to self-organizing systems in life sciences. Without intending to review the vast literature on the subject, we refer the reader to \cite{Burger,Klar,Machado,Med,PT,Perthame} ad the references therein for an introduction. Within this framework, in the present paper we develop a kinetic model for the dynamics of the statistical distributions of the  densities of muscle cells and immune cells in a large cohort of MD patients. Building upon~\cite{DAC}, we focus on a simplified scenario wherein normal muscle cell degeneration is driven by cytotoxic T lymphocytes, the recruitment and activity of which are mediated by macrophages. Next, considering an appropriately scaled version of this kinetic model, we formally derive, as the corresponding mean-field limit, a system of Fokker-Planck equations, from which we subsequently derive, as a  macroscopic model counterpart, a system of differential equations for the mean densities of muscle and immune cells in the cohort of patients and the related variances. Then, building on the approach introduced in \cite{TZ}, and later extended in \cite{BMTZ,MTZ}, we study long-time asymptotics for the mean-field model by determining the quasi-equilibrium cell distribution functions, which are in the form of probability density functions of inverse Gamma distributions, and proving the long-time convergence to such quasi-equilibrium distributions, in an appropriate Sobolev norm. Finally, the analytical results obtained are illustrated by means of a sample of results of numerical simulations. 

In contrast to compartmental models formulated as ordinary differential equations, the kinetic and mean-field models presented here makes it possible to capture emergent macroscopic dynamics resulting from microscopic processes taking place at the cellular level. Moreover, beyond theoretical studies, the proposed modeling approach offers a natural setting to integrate medical information retrieved from heterogeneous cohorts of MD patients, for example extracted from magnetic resonance imaging data, see e.g. \cite{Agosti,Cabini, CTZ}. In fact, by aggregating data from across a cohort of patients, one may reconstruct statistical descriptors of the underlying cell dynamics, which can in turn be used to forecast disease progression at the patient cohort level. In this respect, the kinetic model presented here and its mean-field limit represent the first modeling approach offering the potential of combining a mathematical model of cell dynamics with heterogeneous patient data, so as to provide predictive insights into MD progression.

The remainder of the paper is organized as follows. In Section \ref{sect:2}, we present the kinetic model, along with the underlying microscopic rules modeling cell dynamics, and we then derive the mean-field and macroscopic counterparts of this model. In Section \ref{sect:3}, we study the long-time behaviour of the mean-field model. In Section~\ref{sect:4}, we present a sample of results of numerical simulations through
which we illustrate the theoretical results obtained in the previous sections. In Section~\ref{sect:5}, we conclude with a summary of key findings and a brief discussion of future directions for the modeling approach proposed here.

\section{Kinetic and mean-field models}
\label{sect:2}
Building upon~\cite{DAC}, for simplicity we group the muscle fibers into normal and damaged cells and, among immune cells, we take into account the action of macrophages and cytotoxic T lymphocytes only. We thus consider four populations of cells, which are key players in the dystrophic progression: normal cells, $N$, damaged cells, $D$, macrophages, $M$, and cytotoxic T lymphocytes, $C$. Note that, to distil the essence of the problem, compared to~\cite{DAC} we do not model explicitly the dynamics of helper T cells and regenerating muscle cells.

For a given large cohort of patients under consideration, we denote by $f_J(x,t)$ the statistical distribution of the (number) density of cells of type $J \in \{N,D,M,C\}$ over the patients -- i.e. $f_J(x,t) \, dx$ is the fraction of cohort patients whose density of cells of type $J$ falls within the interval $[x,x+dx) \subset \mathbb{R}_+$ at time $t\ge0$. Naturally, the distribution function $f_J$ is such that 
\begin{equation}
\label{eq:unitmass}
\int_{\mathbb R_+}f_J(x,t)dx = 1 \quad \forall t \geq 0, \quad J \in \{N,D,M,C\}.
\end{equation}
We then define the moment of order $k\ge0$ of the distribution function $f_J$ as
\[
m_J^{(k)}(t) = \int_{\mathbb R_+}x^k f_J(x,t)dx,
\]
and, for ease of notation, we denote the moment of order $k=1$, i.e. the mean, as
\begin{equation}
\label{eq:mJ}
m_J(t) = m_J^{(1)}(t) = \int_{\mathbb R_+}x^k f_J(x,t)dx, 
\end{equation}
and the variance as
\begin{equation}
\label{eq:VJ}
V_J(t) = m_J^{(2)}-(m_J)^2 = \int_{\mathbb R_+}(x-m_J)^2f_J(x,t)dx.
\end{equation}
We refer the interested reader to \cite{DPTZ} for a related modeling approach employed in a different biological context. 

\subsection{Microscopic rules underlying the kinetic model}
\label{sec:microrules}
We now summarise the microscopic rules which underpin the system of integro-differential equations governing the evolution of the distribution functions $f_J(x,t)$, $J \in \{N,D,M,C\}$, that is, the kinetic model~\eqref{eq:system} presented in Section~\ref{sec:kineticmodel}.

\paragraph{Microscopic rules underlying the dynamics of normal and damaged cells.} Building upon~\cite{DAC}, we assume normal cells to become damaged as a result of interactions with cytotoxic T lymphocytes, while interactions between damaged cells and macrophages result in damaged cells being cleared out and normal cells being replenished through muscle regeneration. Therefore, denoting the values of the densities of normal and damaged cells by $x_N \in \mathbb R_+$ and $x_D \in \mathbb R_+$, respectively, and the values of the densities of macrophages and cytotoxic T lymphocytes by $x_M \in \mathbb R_+$ and $x_C \in \mathbb R_+$, respectively, we define the value of the density of normal cells upon interactions between normal cells and cytotoxic T lymphocytes, $x_N^\prime \in \mathbb R_+$, and the value of the density of damaged cells upon interactions between damaged cells and macrophages, $x^\prime_D \in \mathbb R_+$, as
\begin{equation}
\label{eq:xpyp}
\begin{split}
x_N^\prime &= x_N  -  \beta_N\Phi_N(x_C) x_N + \eta_N(x_C) x_N, \\ 
x_D^\prime &= x_D - \beta_D \Phi_D(x_M) x_D + \eta_D(x_M) x_D.
\end{split}
\end{equation}
In the definitions~\eqref{eq:xpyp}, the terms $\beta_N\Phi_N(x_C) x_N$ and $\eta_N(x_C) x_N$ model, respectively, deterministic and random variations in the value of the density of normal cells due to interactions with cytotoxic T lymphocytes. Similarly, the terms $\beta_D \Phi_D(x_M) x_D$ and $\eta_D(x_M) x_D$ model deterministic and random variations in the value of the density of damaged cells due to interactions with macrophages. Here $\Phi_N(x_C)$ and $\Phi_D(x_M)$ are non-negative saturating functions such that $\Phi_N(0)=\Phi_D(0)=0$. To retain simplicity and maintain generality, throughout the paper we set
\begin{equation}
\label{eq:PhiNPhiD}
\Phi_N(x) = \Phi_D(x) = \dfrac{x}{1+x}, \quad x \in \mathbb{R}_+. 
\end{equation}
The parameter $\beta_N>0$ thus provides a possible measure of the maximum relative change in normal cell density due to interactions with T lymphocytes, while the parameter $\beta_D>0$ provides a possible measure of the maximum relative change in damaged cell density due to interactions with macrophages. Moreover, $\eta_N$ and $\eta_C$ are independent random variables such that, denoting the expectation of a random variable with respect to the corresponding probability density function by $\left\langle \cdot \right\rangle$, 
\[
\left\langle \eta_N \right\rangle = \left\langle \eta_D \right\rangle = 0, \qquad \left\langle \eta_N^2 \right\rangle = \sigma_N^2 \Phi_N(x_C), \qquad \left\langle \eta_D^2 \right\rangle = \sigma_D^2 \Phi_D(x_M).
\]
Hence, the parameters $\sigma_N$ and $\sigma_D$ are linked to the standard deviations of the zero-mean random variables that model fluctuations in the densities of normal and damaged cells.

Coherently with the microscopic rules defined via~\eqref{eq:xpyp}, we define the value of the density of normal cells upon interactions between damaged cells and macrophages, $x_N^{\prime\prime} \in \mathbb R_+$, and the value of the density of damaged cells upon interactions between normal cells and cytotoxic T lymphocytes, $x_D^{\prime\prime} \in \mathbb R_+$, as
\begin{equation}
\label{eq:xppypp}
\begin{split}
x_N^{\prime\prime} &= x_N + \beta_D \Phi_D(x_M) x_D, \\
x_D^{\prime\prime} &= x_D + \beta_N\Phi_N(x_C) x_N.
\end{split}
\end{equation}

Note that the definitions~\eqref{eq:xpyp} and~\eqref{eq:xppypp} complemented with the definitions~\eqref{eq:PhiNPhiD} are such that if $x_N, x_D, x_M, x_C\ge0$ then both  $x_N^\prime, x_D^\prime\ge0$ and $x_N^{\prime\prime}, x_D^{\prime\prime}\ge0$ for any $\beta_N, \beta_D>0$. 

\paragraph{Microscopic rules underlying the dynamics of cytotoxic T lymphocytes and macrophages.} We let the densities of macrophages and cytotoxic T lymphocytes, $x_M \in \mathbb R_+$ and $x_C \in \mathbb R_+$, undergo natural decay at constant rates $\beta_M>0$ and $\beta_C > 0$, respectively, along with small random fluctuations due to homeostatic regulation mechanisms. Moreover, building again upon~\cite{DAC}, we let the density of macrophages increase at a rate proportional to the density of damaged cells, $x_D \in \mathbb R_+$, with a constant of proportionality $\gamma_M>0$, while the density of cytotoxic T lymphocytes increases at a rate proportional to the the density of macrophages with a constant of proportionality $\gamma_C>0$. Hence, we define the value of the density of macrophages (resp. cytotoxic T lymphocytes) upon natural decay with small random fluctuations, $x_M^\prime \in \mathbb R_+$ (resp. $x_C^\prime \in \mathbb R_+$), and the value of the density of macrophages (resp. cytotoxic T lymphocytes) upon growth caused by the presence of damaged cells (resp. macrophages),  $x_M^{\prime\prime}$ (resp. $x_C^{\prime\prime}$), as
\begin{equation}
\label{eq:m}
\begin{split}
x_M^\prime &= x_M - \beta_M x_M + \eta_M x_M,  \\
x_M^{\prime\prime} &= x_M + \gamma_M x_D,
\end{split}
\end{equation}
and 
\begin{equation}
\label{eq:c}
\begin{split}
x_C^\prime &= x_C - \beta_C x_C + \eta_C x_C, \\
x_C^{\prime\prime} &= x_C + \gamma_C x_M.
\end{split}
\end{equation}
Here $\eta_M$ and $\eta_C$ are independent random variables such that
\[
\left\langle \eta_M \right\rangle = \left\langle \eta_C \right\rangle = 0, \qquad \left\langle \eta_M^2 \right\rangle = \sigma_M^2, \qquad \left\langle \eta_C^2 \right\rangle = \sigma_C^2.
\]
Therefore, the parameters $\sigma_M$ and $\sigma_C$ are the standard deviations of the zero-mean random variables that model fluctuations in the densities of macrophages and cytotoxic T lymphocytes.

Note again that the definitions~\eqref{eq:m} and~\eqref{eq:c} are such that if $x_N, x_D, x_M, x_C\ge0$ then both $x_M^\prime, x_C^\prime\ge0$ and $x_M^{\prime\prime}, x_C^{\prime\prime}\ge0$ for any $\beta_M, \beta_C>0$. 

\subsection{The kinetic model}
\label{sec:kineticmodel}
The evolution of the distribution functions $f_J(x,t)$, $J \in \{N,D,C,M\}$, is governed by the kinetic model defined by the following system of integro-differential equations
\begin{equation}
\label{eq:system}
\begin{split}
\dfrac{\partial}{\partial t}f_N &= \mathcal Q_{N}(f_N,f_C) + \mathcal F^{\beta_D}(f_N,f_D,f_M),\\
\dfrac{\partial}{\partial t}f_D &= \mathcal Q_{D}(f_D,f_M) + \mathcal F^{\beta_N}(f_D,f_N,f_C), \\
\dfrac{\partial}{\partial t}f_M &= \mathcal L(f_M) + \mathcal R^{\gamma_M}(f_M,f_D), \\
\dfrac{\partial}{\partial t}f_C &= \mathcal L(f_C) + \mathcal R^{\gamma_C}(f_C,f_M),
 \end{split}
 \qquad (x,t) \in \mathbb{R}^2_+,
\end{equation}
where the microscopic rules formulated in Section~\ref{sec:microrules} are mirrored into the definitions of the operators $\mathcal Q$, $\mathcal F$, $\mathcal L$, and $\mathcal R$, as it is common in kinetic theory -- see e.g. \cite{PT} and references therein. 

In the system~\eqref{eq:system}: 
\begin{itemize}
\item the operator $\mathcal Q_H(f_H,f_K)$ is defined as
\[
\begin{split}
\mathcal Q_H(f_H,f_K)(x,t) = \left\langle \int_{\mathbb R_+}\kappa(x_*) \left( \dfrac{1}{{}^\prime J_{HK}}f_H({}^\prime x,t)f_K({}^\prime x_*,t) - f_H(x,t)f_K(x_*,t) \right)dx_* \right\rangle, \\
\end{split}
\]
where $\kappa$ is an interaction kernel that, along the lines of~\cite{BMTZ}, is defined as 
\begin{equation}
\label{eq:defsmallk}
\kappa(x) = 1 + x, \quad x \in \mathbb{R}_+, 
\end{equation}
while ${}^\prime J_{HK}$ is the determinant of the Jacobian associated with the transformation ${}^\prime x\to x$ defined via the first line of~\eqref{eq:xpyp} for $(H,K) = (N,C)$ and the second line of~\eqref{eq:xpyp} for $(H,K) = (D,M)$;

\item the operator $\mathcal F^{\beta_K}(f_H,f_K,f_L)$ is defined as
\[\begin{split}
\mathcal F^{\beta_K}(f_H,f_K,f_L)(x,t) = &  \int_{\mathbb R^2_+} \kappa(x^*) \Big(\dfrac{1}{{}^{\prime\prime} J_{HKL}}f_H({}^{\prime\prime} x,t)f_K( x_*,t ) f_L( x^*,t ) 
\\ &- f_H(x,t) f_K(x_*,t) f_L(x^*,t)\Big) dx_* dx^*, 
\end{split}\]
where ${}^{\prime\prime} J_{HKL}$ is the determinant of the Jacobian associated with the transformation ${}^{\prime\prime} x\to x$ defined via the first line of~\eqref{eq:xppypp} for $(H,K,L) = (N,D,M)$ and the second line of~\eqref{eq:xppypp} for $(H,K,L) = (D,N,C)$; 
\item the operator $\mathcal L(f_H)$ is defined as
\[
\mathcal L(f_H)(x,t)  = \left\langle  \dfrac{1}{{}^\prime J_H} f_H({}^\prime x,t) \right\rangle - f_H(x,t), 
\]
where ${}^{\prime} J_{H}$ is the Jacobian associated with the transformation ${}^\prime x\to x$ defined via the first line of~\eqref{eq:m} for $H = M$ and the first line of~\eqref{eq:c} for $H = C$;
\item the operator $\mathcal R^{\gamma_H}(f_H,f_K)$ is defined as
\[
\mathcal R^{\gamma_H}(f_H,f_K)(x,t) = \int_{\mathbb R_+} \left( \dfrac{1}{{}^{\prime\prime} J_{HK}}f_H({}^{\prime\prime} x,t)f_K({}^{\prime\prime} x_*,t ) - f_H(x,t) f_K(x_*,t)\right) dx_*, 
\]
where ${}^{\prime\prime} J_{HK}$ is the determinant of the Jacobian associated with the transformation ${}^{\prime\prime} x\to x$ defined via the second line of~\eqref{eq:m} for $(H,K) = (M,D)$ and the second line of~\eqref{eq:c} for $(H,K) = (C,M)$.
\end{itemize} 

Note that, as expected, when subject to initial conditions with non-negative components of unit mass, i.e. initial conditions such that
\begin{equation}
\label{eq:unitmassIC}
\int_{\mathbb R_+}f_J(x,0)dx = 1, \quad J \in \{N,D,M,C\},
\end{equation}
the components of the solution to the kinetic model~\eqref{eq:system} are also non-negative and with unit mass for all $t>0$, that is, conditions~\eqref{eq:unitmass} hold.

\subsubsection{Evolution of key observable macroscopic quantities for the kinetic model}
\label{sec:macrokin}
To investigate the dynamics of the cell populations in the cohort of patients under consideration at the level of observable macroscopic quantities, such as the mean, $m_J(t)$, defined via~\eqref{eq:mJ} (i.e. the expected value of the density of cells in population $J$ at time $t$) and the variance, $V_J(t)$, defined via~\eqref{eq:VJ}, it is convenient to rewrite the kinetic model~\eqref{eq:system} in the following weak form
\begin{equation}
\label{eq:kinetic}
\begin{split}
\dfrac{d}{dt} \int_{\mathbb R_+} \varphi_N(x)f_N(x,t)dx = &\int_{\mathbb R_+^2} \kappa(c)\left\langle \varphi_N(x^\prime)-\varphi_N(x) \right\rangle f_N(x,t) f_C(c,t) dc\,dx \, + \\
& \int_{\mathbb R_+^3} \kappa(z) (\varphi_N(x^{\prime\prime})-\varphi_N(x))f_N(x,t)f_M(z,t)f_D(y,t)dy\,dz\,dx, \\
\dfrac{d}{dt}  \int_{\mathbb R_+} \varphi_D(x)f_D(x,t)dx = &\int_{\mathbb R_+^2} \kappa(z)\left\langle \varphi_D(x^\prime)-\varphi_D(x) \right\rangle f_D(x,t) f_M(z,t) dz\,dx \, + \\
& \int_{\mathbb R_+^3} \kappa(c)(\varphi_D(x^{\prime\prime})-\varphi_D(x))f_D(x,t)f_C(c,t)f_N(n,t)dn\,dc\,dx, \\
\dfrac{d}{dt} \int_{\mathbb R_+} \varphi_M(x)f_M(x,t)dx = &\int_{\mathbb R_+} \left\langle \varphi_M(x^\prime)-\varphi_M(x) \right\rangle f_M(x,t) dx \, + \\
& \int_{\mathbb R_+^2} (\varphi_M(x^{\prime\prime})-\varphi_M(x))f_D(y,t) f_M(x,t)dy\,dx, \\
\dfrac{d}{dt} \int_{\mathbb R_+} \varphi_C(x)f_C(x,t)dx = &\int_{\mathbb R_+} \left\langle \varphi_C(x^\prime)-\varphi_C(x) \right\rangle f_C(x,t) dx \, + \\
& \int_{\mathbb R_+^2} (\varphi_C(x^{\prime\prime})-\varphi_C(x))f_C(x,t) f_M(z,t)dz \,dx,
\end{split}
\qquad
\end{equation}
where $\varphi_J(x)$, $J \in\{N,D,M,C\}$, are smooth test functions and the interaction kernel $\kappa$ is defined via~\eqref{eq:defsmallk}.

\paragraph{Evolution equations for $m_J(t)$, $J \in \{N,D,M,C\}$.} Choosing $\varphi_J(x) = x$ in the weak form~\eqref{eq:kinetic} of the kinetic model~\eqref{eq:system}, we obtain the following system of differential equations for the means  
\begin{equation}
\label{eq:mean_K_NDMC}
\begin{split}
\dfrac{d}{dt}m_N(t) =& -\beta_N m_N(t) m_C(t)  + \beta_D m_M(t) m_D(t), \\
\dfrac{d}{dt}m_D(t) =& -\beta_D m_D(t) m_M(t)  + \beta_N m_N(t) m_C(t), \\
\dfrac{d}{dt}m_M(t) =& -\beta_M m_M(t) + \gamma_M m_D(t), \\
\dfrac{d}{dt}m_C(t) =& -\beta_C m_C(t)  + \gamma_C m_M(t).
\end{split}
\end{equation}
Note that the system~\eqref{eq:mean_K_NDMC} is consistent with the compartmental model considered in~\cite{DAC} within the framework of the simplifying assumptions that underlie the kinetic model considered here. 

From the system~\eqref{eq:mean_K_NDMC} we see that when the kinetic model~\eqref{eq:system} is subject to non-negative initial conditions that, in addition to~\eqref{eq:unitmassIC}, satisfy  
\begin{equation}
\label{eq:boundedmeanIC}
0 < m_J(0) = \int_{\mathbb R_+}x f_J(x,0)dx < \infty, \quad J \in \{N,D,M,C\},
\end{equation}
then
\begin{equation}
\label{eq:boundedmean}
0 < m_J(t) = \int_{\mathbb R_+}x f_J(x,0)dx < \infty \quad \forall t > 0, \quad J \in \{N,D,M,C\}.
\end{equation}
From the system~\eqref{eq:mean_K_NDMC} we also see that 
$$
m_N(t) + m_D(t) = m_N(0) + m_D(0) \quad \forall t > 0,
$$
and, therefore, under an initial condition that satisfies the additional condition
\begin{equation}
\label{eq:ICNpmD}
m_N(0) + m_D(0) = m^0 > 0
\end{equation}
we have
\begin{equation}
\label{eq:NpmD}
m_N(t) + m_D(t) = m^0 \quad \forall t \geq 0.
\end{equation}
Finally, under assumptions~\eqref{eq:boundedmeanIC} and~\eqref{eq:ICNpmD} on the initial condition, as $t \to \infty$ the solution of the system~\eqref{eq:mean_K_NDMC} converges to the equilibrium of components 
\begin{equation}
\label{eq:MNDMCinfty}
m_N^\infty = \dfrac{\beta_D\beta_C m^0}{\beta_D \beta_C + \beta_N\gamma_C}, \quad m_D^\infty = \dfrac{\beta_N \gamma_C m^0}{\beta_D \beta_C + \beta_N \gamma_C}, \quad m_M^\infty = \dfrac{\gamma_M}{\beta_M} m_D^\infty, \quad m_C^\infty = \dfrac{\gamma_C\gamma_M}{\beta_C\beta_M} m_D^\infty. 
\end{equation}
This is also shown by the numerical solutions displayed in the top, left panel and the center panels of Figure 1.

\paragraph{Evolution equations for $V_J(t)$, $J \in \{N,D,M,C\}$.} Choosing $\varphi_J(x) = (x-m_J)^2$ in the weak form~\eqref{eq:kinetic} of the system~\eqref{eq:system}, we find the following system of differential equations for the variances  
\begin{equation}
\label{eq:var_K_NDMC}
\begin{split}
\dfrac{d}{dt} V_N(t) =& \beta_N^2\int_{\mathbb R_+^2 } \dfrac{n^2 c^2}{1+c}f_N(x,t)f_C(c,t)dc\,dx +  \beta_D^2 \left[V_M(t) + m_M^2(t)\right] \left[ V_D(t) + m_D^2(t)\right] + \\
& - (2\beta_N - \sigma^2_N) m_C(t) \left[V_N(t)  - \dfrac{\sigma_N^2m_N^2(t)}{2\beta_N - \sigma^2_N}\right], \\
\dfrac{d}{dt} V_D(t) =&\beta_D^2 \int_{\mathbb R_+^2}\dfrac{z^2 x^2}{1+z}f_D(x,t)f_M(z,t)dz\,dx + \beta_N^2 \left[V_N(t) + m_N^2(t)\right]   \left[V_C(t) + m_C^2(t)\right] + \\
& - (2\beta_D - \sigma_D^2) m_M(t)\left[V_D(t) - \dfrac{\sigma_D^2m_D^2(t)}{2\beta_D - \sigma^2_D} \right], \\
\dfrac{d}{dt} V_M(t) =&\beta_M^2 \left[V_M(t) + m_M^2\right] - (2\beta_M-\sigma_M^2) \left[V_M(t) - \dfrac{\sigma_M^2(t)m_M^2(t)}{2\beta_M-\sigma_M^2}\right]  + \gamma_M^2 [V_D(t) + m_D^2(t)],\\
\dfrac{d}{dt}V_C(t) =& \beta_C^2 [V_C(t) + m_C^2(t)] - (2\beta_C - \sigma_C^2) \left[V_C(t) - \dfrac{\sigma_C^2 m_C^2(t)}{2\beta_C - \sigma_C^2}\right] + \gamma_C^2 [V_M(t) + m_M^2(t)]. 
\end{split}
\end{equation}
Note that, in contrast to the system~\eqref{eq:mean_K_NDMC}, the system~\eqref{eq:var_K_NDMC} is not closed as it still requires the knowledge of the distribution functions $f_J(x,t)$, $J \in \{N,D,M,C\}$. This hinders obtaining a self-consistent representation of the dynamics of the cell populations in the patient cohort in terms of observable macroscopic quantities. To circumvent this difficulty, in the next section we will derive a mean-field counterpart of the kinetic model~\eqref{eq:system} -- cf. the mean-field model~\eqref{eq:MFN}-\eqref{eq:MFoBCs} -- for which the corresponding system of differential equations for the means is the same as~\eqref{eq:mean_K_NDMC}, while the system of differential equations for the variances is closed. In addition, for the mean-field model one can also obtain a detailed characterisation of the long-term limits of the distribution functions, and not only of the corresponding means and variances, as demonstrated in Section~\ref{sect:3}, which supports a more detailed statistical representation of the cell populations in the patient cohort.

\subsection{Corresponding mean-field model} \label{sect:mean-field}
Building upon the method employed, for instance, in~\cite{BMTZ,DPTZ,FPTT}, we introduce a small scaling parameter $0 < \epsilon \ll 1$ and scale the time variable as 
\begin{equation}
\label{eq:timescalingeps}
\begin{split}
t \to t/\epsilon
\end{split}
\end{equation}
and the model parameters as 
\begin{equation}
\label{eq:paramscalingeps}
\begin{split}
&\beta_J  \to \epsilon \beta_J \;\; \text{and} \;\; \sigma_J^2 \to \epsilon \sigma_J^2 \;\; \text{for } J \in \{N,D,M,C\}, \qquad \gamma_C\to \epsilon \gamma_C, \qquad \gamma_M \to \epsilon \gamma_M.
\end{split}
\end{equation}
Under this parameter scaling, the dynamics governed by the microscopic rules defined via~\eqref{eq:xpyp}, \eqref{eq:xppypp}, \eqref{eq:m}, and~\eqref{eq:c} become quasi-invariant, which allows us to resort to the following Taylor expansions
\[
\begin{split}
\varphi_J(x^\prime)- \varphi_J(x) &\approx (x^\prime-x)\varphi'_J(x) + \dfrac{1}{2}(x^\prime-x)^2 \varphi_J^{\prime\prime}(x) + \dfrac{1}{6}(x^\prime-x)^3 \varphi^{\prime\prime\prime}_J(\bar x), \\
\varphi_J(x^{\prime\prime})- \varphi_J(x) &\approx \varphi_J^\prime(x)  + \dfrac{1}{2}\varphi_J^{\prime\prime}(\tilde x)(x^{\prime\prime}-x)^2,
\end{split}\]
for $J \in \{N,D,M,C\}$, where 
\[
\bar x \in (\min\{ x,x^\prime\}, \max\{x,x^\prime\}), \qquad  \tilde x \in (\min\{ x,x^{\prime\prime}\}, \max\{x,x^{\prime\prime}\}). 
\]
Then, introducing the scaled distribution functions 
$$
f_J^\epsilon(x,t) = f_J(x,t/\epsilon), \quad J \in \{N,D,M,C\},
$$
substituting the time scaling~\eqref{eq:timescalingeps} along with the above Taylor expansions into the weak form~\eqref{eq:kinetic} of the integro-differential equation for the distribution function $f_N(x,t)$, we formally obtain the following  weak form of the integro-differential equation for the scaled distribution function $f_N^\epsilon(x,t)$ 
\begin{equation}
\label{eq:wfidef_Neps}
\begin{split}
\dfrac{d}{dt}  \int_{\mathbb R_+}\varphi_N(x)f_N^\epsilon(x,t)dx =&   - \beta_N\int_{\mathbb R_+^2}x\,c\,\varphi_N^\prime(x) f_N^\epsilon(x,t)f_C^\epsilon(c,t)dc\,dx + \\
&\dfrac{\sigma_N^2}{2} \int_{\mathbb R_+^2}x^2\,c\,\varphi_N^{\prime\prime}(x) f_N^\epsilon(x,t)f_C^\epsilon(c,t)dc\,dx + \\
&   \beta_D\int_{\mathbb R_+} z\,y\, \varphi_N^\prime(x)f_N^\epsilon(x,t) f_D^\epsilon(y,t)f^\epsilon_M(z,t)dz\,dy\,dx + \dfrac{R^\epsilon_N}{\epsilon},
\end{split}
\end{equation}
where
\[
\begin{split}
R^\epsilon_N =& \dfrac{1}{6} \int_{\mathbb R_+^2}\kappa(c) \left\langle \left(- \epsilon \beta_N \dfrac{xc}{1+c} + x\eta_N(c)\right)^3 \right\rangle\varphi_N^{\prime\prime\prime}(x) f_N^\epsilon(x,t)f_C^\epsilon(c,t)dc\,dx \\
& + \dfrac{\epsilon^2\beta_D^2}{2}\int_{\mathbb R_+}  z\,y\, \varphi^{\prime\prime}_N(x)f_N^\epsilon(x,t) f^\epsilon_M(z,t) f^\epsilon_D(y,t) dy\,dz\,dx. 
\end{split}\]
Provided that the third order moment of the probability density function of the random variable $\eta_N$ is bounded, we can rewrite $\eta_N = \sqrt{\dfrac{\epsilon \sigma_N^2c}{1+c} }\hat \eta$, with $\left\langle \hat \eta \right\rangle = 0$ and $\left\langle \hat \eta^2 \right\rangle = 1$, from which it follows that $R^\epsilon_N \approx \epsilon^3 + \epsilon^2 + \epsilon \sqrt{\epsilon}$. Moreover, under the time scaling~\eqref{eq:timescalingeps} and the parameter scaling~\eqref{eq:paramscalingeps}, the means $m_J^\epsilon(t) = m_J(x,t/\epsilon)$, $J \in \{N,D,M,C\}$, defined via~\eqref{eq:mJ} coincide with the components of the solution to the system~\eqref{eq:mean_K_NDMC} (i.e. if $m_J^\epsilon(0) = m_J(0)$ then $m_J^\epsilon(t) = m_J(t)$ for all $t>0$), which are bounded (cf. \eqref{eq:boundedmean}). Hence, we have that $|R^\epsilon_N|/\epsilon \to 0$ for $\epsilon \to 0^+$. As a result, denoting by $f_N(x,t)$ the weak limit of $f_N^\epsilon(x,t)$ as $\epsilon \to 0^+$, from~\eqref{eq:wfidef_Neps} we formally obtain, in the asymptotic regime $\epsilon \to 0^+$, the following weak formulation of the governing equation for $f_N(x,t)$ 
\[
\begin{split}
\dfrac{d}{dt} \int_{\mathbb R_+}\varphi_N(x)f_N(x,t)dx =& -\beta_N \int_{\mathbb R_+} x m_C(\tau)\varphi_N^\prime(x) f_N(x,t)dx + \dfrac{\sigma_N^2m_C(t)}{2} \int_{\mathbb R_+}x^2 \varphi_N^{\prime\prime}(x) f_N(x,t)dx \\
& + \beta_D m_D(t)m_M(t) \int_{\mathbb R_+}\varphi_N^\prime(x) f_N(x,t)dx, 
\end{split}\]
from which, integrating back by parts, we find
\[
\begin{split}
\dfrac{d}{dt}\int_{\mathbb R_+}\varphi_N(x)f_N(x,t)dx  =& \int_{\mathbb R_+} \dfrac{\partial}{\partial x} \left[(\beta_N m_C(t)x - \beta_D m_M(t)m_D(t)) f_N(x,t)\right] \varphi_N(x)dx + \\
&\dfrac{\sigma^2_N m_C(t)}{2} \int_{\mathbb R_+} \dfrac{\partial^2}{\partial x^2}(x^2 f_N(x,t)) \varphi_N(x)dx.
\end{split}\]
This is a weak formulation of the following Fokker-Planck equation
\begin{equation}
	\label{eq:MFN}
\dfrac{\partial}{\partial t} f_N(x,t) = \dfrac{\partial}{\partial x}  \left[ (\beta_N  m_C(t)x - \beta_D m_M(t)m_D(t)) f_N(x,t)\right]  + \dfrac{\sigma_N^2m_C(t)}{2} \dfrac{\partial^2}{\partial x^2}(x^2 f_N(x,t))
\end{equation}
posed on $\mathbb R_+$ and subject to the no-flux boundary condition below
\begin{equation}
	\label{eq:MFNBCs}
 \left[ (\beta_N  m_C(t)x - \beta_D m_M(t)m_D(t)) f_N(x,t)\right]  + \dfrac{\sigma_N^2m_C(t)}{2} \dfrac{\partial}{\partial x}(x^2 f_N(x,t))\Big|_{x= 0} = 0. 
\end{equation}
Proceeding in a similar way, we formally obtain the following Fokker-Planck equations for the weak limits, $f_J(x,t)$, of the the scaled distribution functions, $f_J^\epsilon(x,t)$, for $J \in \{D,M,C\}$
\begin{equation}
\label{eq:MFo}
\begin{split}
\dfrac{\partial}{\partial t} f_D(x,t) &=\dfrac{\partial}{\partial x} \left[ (\beta_D m_M(t) x - \beta_N m_N(t) m_C(t)) f_D(x,t)\right] + \dfrac{\sigma_D^2}{2}m_M(t)\dfrac{\partial^2}{\partial x^2} (x^2 f_D(x,t)), \\
\dfrac{\partial}{\partial t} f_M(x,t) &= \dfrac{\partial}{\partial x} \left[ (\beta_M x - \gamma_M m_D(t))f_M(x,t)\right] +\dfrac{\sigma_M^2}{2}\dfrac{\partial^2}{\partial x^2} (x^2 f_M(x,t)), \\
\dfrac{\partial}{\partial t} f_C(x,t) &=\dfrac{\partial}{\partial x}\left[ (\beta_C x - \gamma_C m_M(t))f_C(x,t)\right] + \dfrac{\sigma_C^2}{2}\dfrac{\partial^2}{\partial x^2} (x^2 f_C(x,t)),
\end{split}
\end{equation}
which are posed on $\mathbb R_+$ and are subject, respectively, to the following no-flux boundary conditions
\begin{equation}
\label{eq:MFoBCs}
\begin{split}
\left[ (\beta_D m_M(t) x - \beta_N m_N(t) m_C(t)) f_D\right] + \dfrac{\sigma_D^2}{2}m_M(t)\dfrac{\partial}{\partial x} (x^2 f_D) \Big|_{x= 0} = 0, \\
 \left[ (\beta_M x - \gamma_M m_D(t))f_M(x,t)\right] +\dfrac{\sigma_M^2}{2}\dfrac{\partial}{\partial x} (x^2 f_M)\Big|_{x= 0} = 0, \\
\left[ (\beta_C x - \gamma_C m_M(t))f_C(x,t)\right] + \dfrac{\sigma_C^2}{2}\dfrac{\partial}{\partial x} (x^2 f_C)\Big|_{x= 0} = 0.
\end{split}
\end{equation}
The system~\eqref{eq:MFN}, \eqref{eq:MFo}, complemented with the definitions~\eqref{eq:mJ}, posed on $\mathbb R_+$ and subject to the boundary conditions~\eqref{eq:MFNBCs}, \eqref{eq:MFoBCs} is a mean-field limit of the kinetic model~\eqref{eq:system}.

Note that, as it is for the kinetic model~\eqref{eq:system}, when subject to initial conditions with non-negative components of unit mass, the components of the solution to the mean-field model~\eqref{eq:MFN}-\eqref{eq:MFoBCs} are also non-negative with unit mass for all $t>0$ -- i.e. conditions~\eqref{eq:unitmass} are met.

\subsubsection{Evolution of key observable macroscopic quantities for the mean-field model}
\label{sec:macromf}
Through direct calculations, one can easily verify that: as expected, the evolution of the means, $m_J(t)$, defined via~\eqref{eq:mJ}, for the mean-field model~\eqref{eq:MFN}-\eqref{eq:MFoBCs} is governed by the system~\eqref{eq:mean_K_NDMC} as it is for the kinetic model~\eqref{eq:system}; on the other hand, the evolution of the variances, $V_J(t)$, defined via~\eqref{eq:VJ} is governed by the following system which, in contrast to the system~\eqref{eq:var_K_NDMC} obtained for the kinetic model~\eqref{eq:system}, is a closed system:
\begin{equation}
\label{eq:var_MF_NDMC}
\begin{split}
\dfrac{d}{dt} V_N(t) &= -(2\beta_N - \sigma_N^2)m_C(t) \left[V_N(t) -\dfrac{ \sigma_N^2 m_N^2(t)}{2\beta_N - \sigma_N^2} \right], \\
\dfrac{d}{dt} V_D(t) &= - (2\beta_D - \sigma_D^2)m_M(t) \left[ V_D(t) - \dfrac{\sigma_D^2 m_D^2(t)}{2\beta_D-\sigma_D^2} \right], \\
\dfrac{d}{dt} V_M(t) &=  - (2\beta_M - \sigma_M^2) \left[V_M(t) - \dfrac{\sigma_M^2 m_M^2(t)}{2\beta_M - \sigma_M^2} \right],\\
\dfrac{d}{dt} V_C(t) &= -(2\beta_C - \sigma_C^2)\left[V_C(t) - \dfrac{\sigma_C^2 m_C^2(t)}{2\beta_C-\sigma_C^2} \right]. 
\end{split}
\end{equation}
When the mean-field model~\eqref{eq:MFN}-\eqref{eq:MFoBCs} is subject to non-negative initial conditions that, in addition to~\eqref{eq:unitmassIC}, \eqref{eq:boundedmeanIC}, and \eqref{eq:ICNpmD}, satisfy  
\begin{equation}
\label{eq:boundedvarIC}
0 < V_J(0) = \int_{\mathbb R_+}(x-m_J(0))^2f_J(x,0)dx < \infty, \quad J \in \{N,D,M,C\},
\end{equation}
under the following additional assumptions on the model parameters
\begin{equation}
\label{eq:asssigmabetaJ}
\sigma_J^2 < 2\beta_J \quad J \in \{N,D,M,C\},
\end{equation}
as $t \to \infty$ the solution of the system~\eqref{eq:var_MF_NDMC} converges to the equilibrium of components 
\begin{equation}
\label{eq:VNDMCinfty}
\begin{split}
V_N^\infty &= \dfrac{\sigma^2_N}{2\beta_N - \sigma_N^2} (m_N^\infty)^2, \qquad V_D^\infty = \dfrac{\sigma_D^2}{2\beta_D-\sigma_D^2} (m_D^\infty)^2, \\
V_M^\infty &=  \dfrac{\sigma_M^2 }{2\beta_M - \sigma_M^2} (m_M^\infty)^2, \qquad V_C^\infty = \dfrac{\sigma_C^2 }{2\beta_C-\sigma_C^2}  (m_C^\infty)^2.
\end{split}
\end{equation}
This is also shown by the numerical solutions displayed in the top, right panel and the bottom panels of Figure~\ref{fig:fig1}.

\begin{figure}[h!]
\centering
\includegraphics[scale = 0.3]{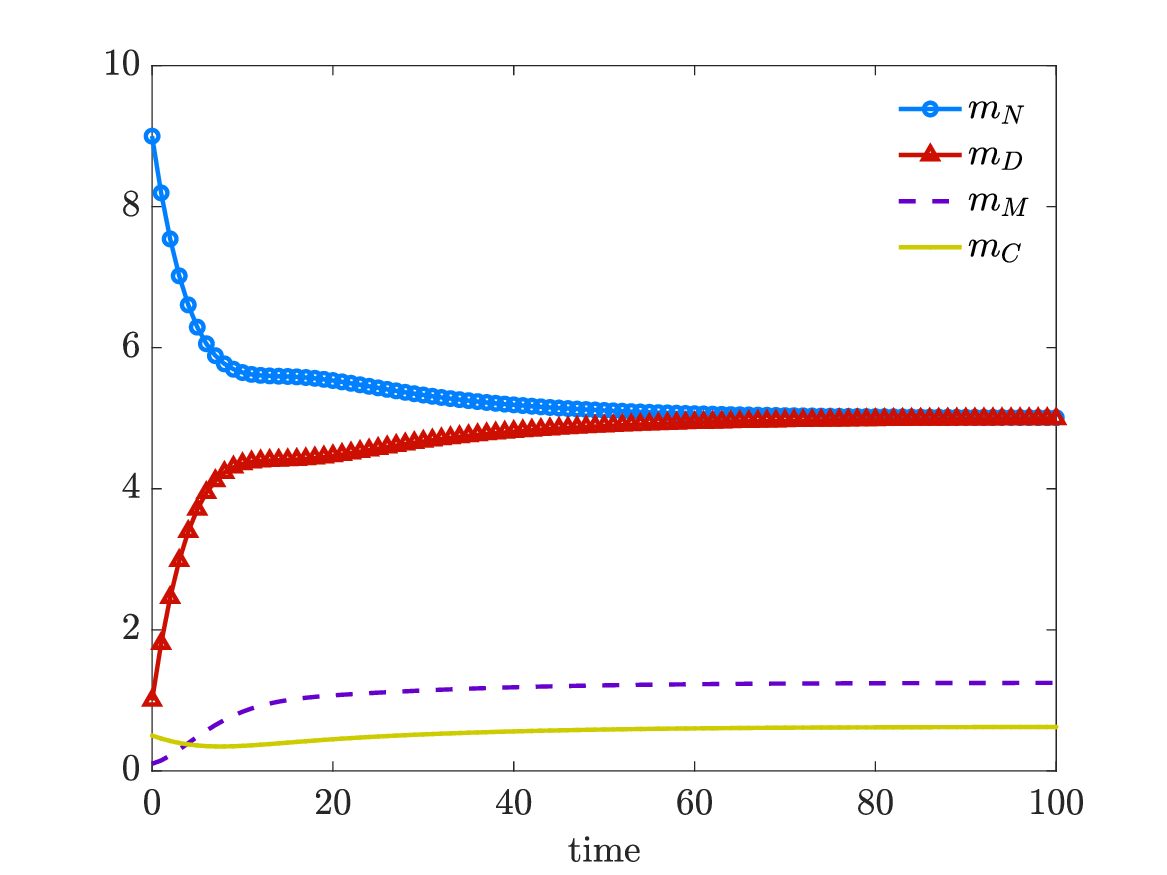}
\includegraphics[scale = 0.3]{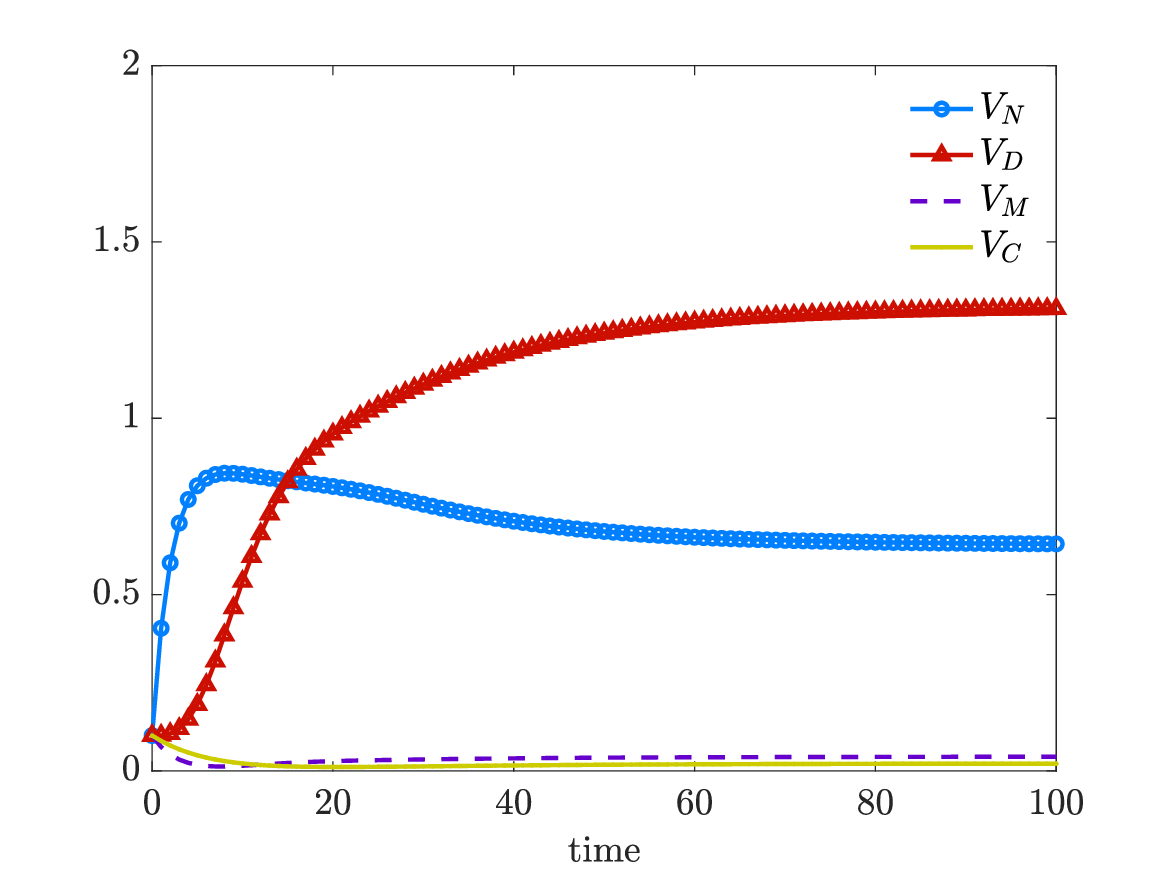}\\
\includegraphics[scale = 0.3]{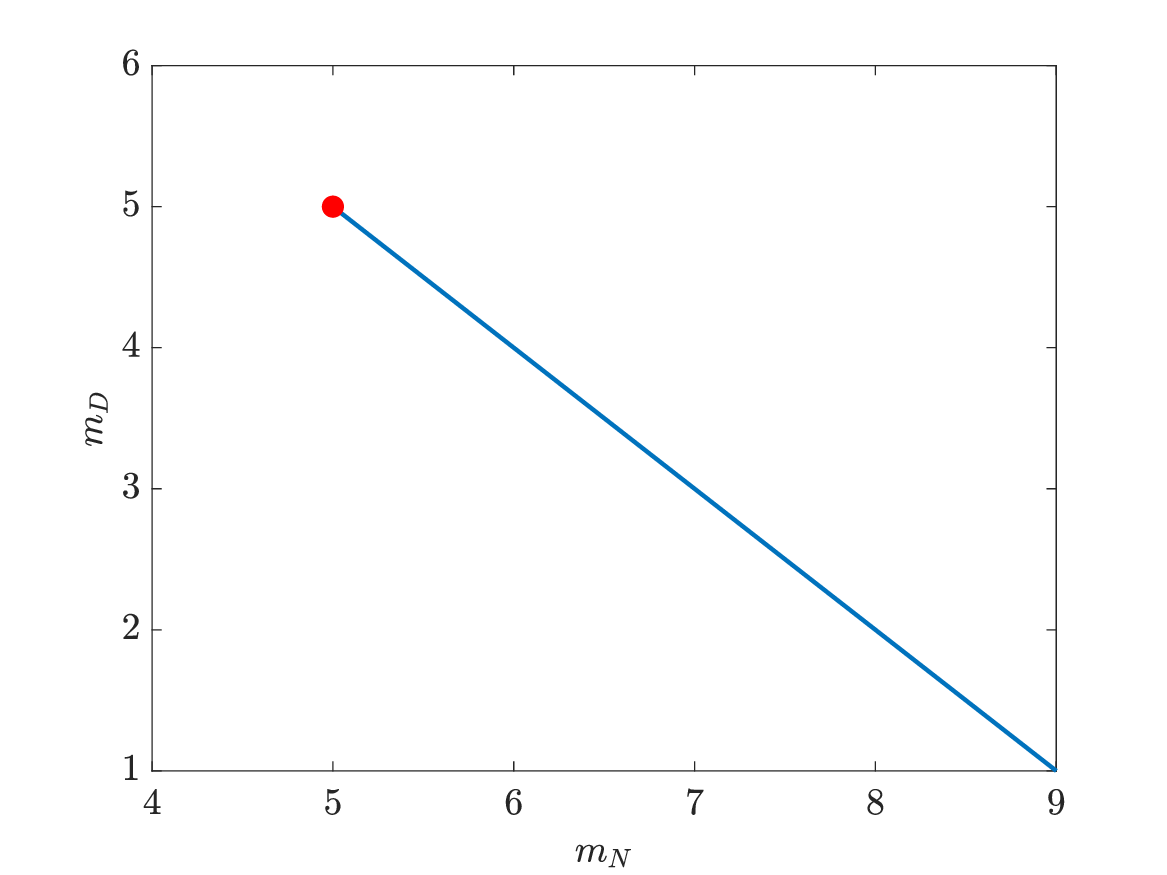}
\includegraphics[scale = 0.3]{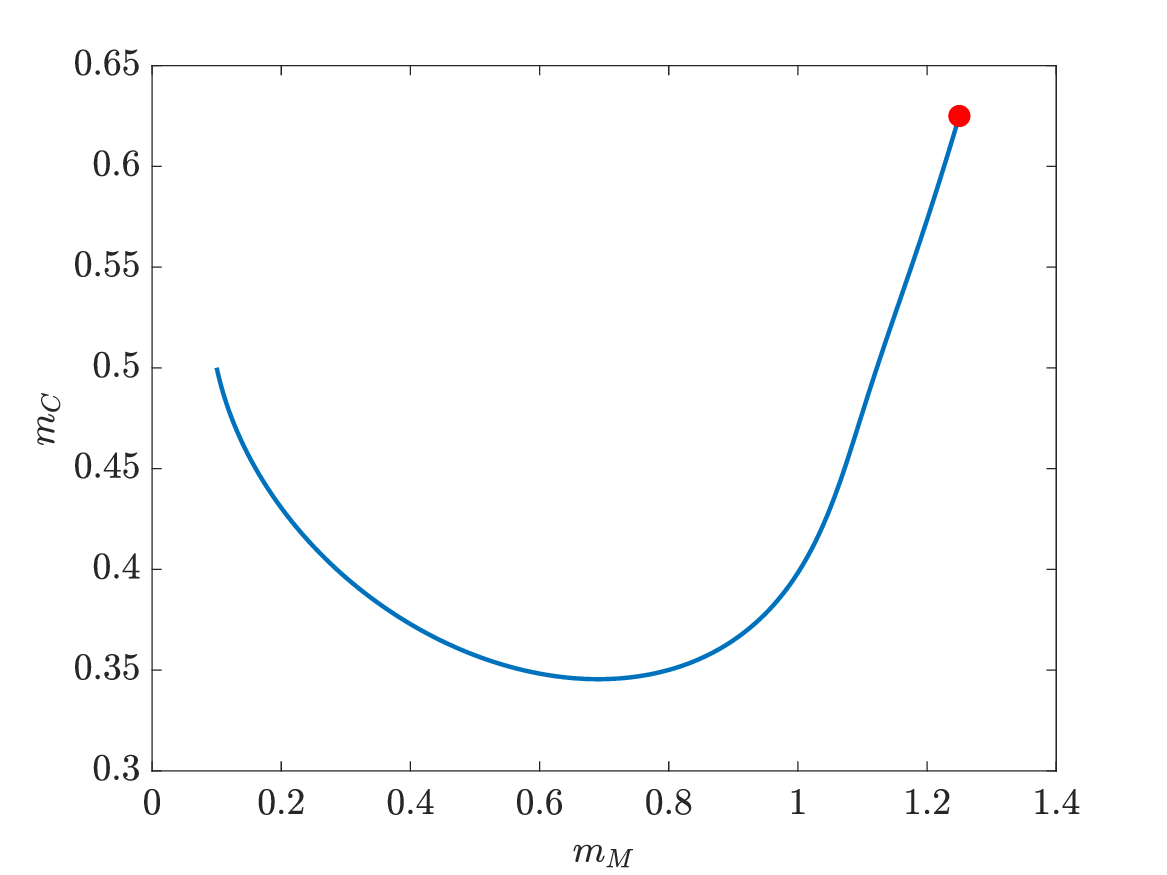}\\
\includegraphics[scale = 0.3]{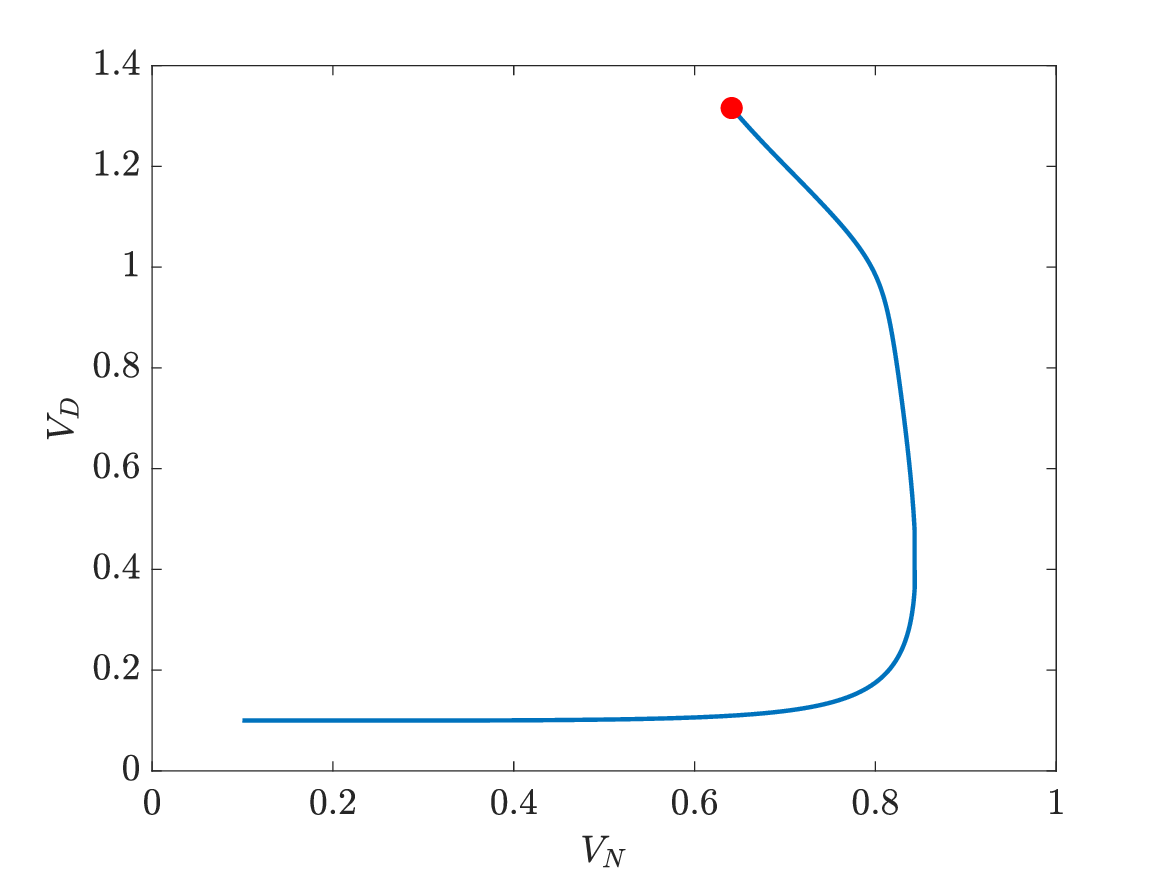}
\includegraphics[scale = 0.3]{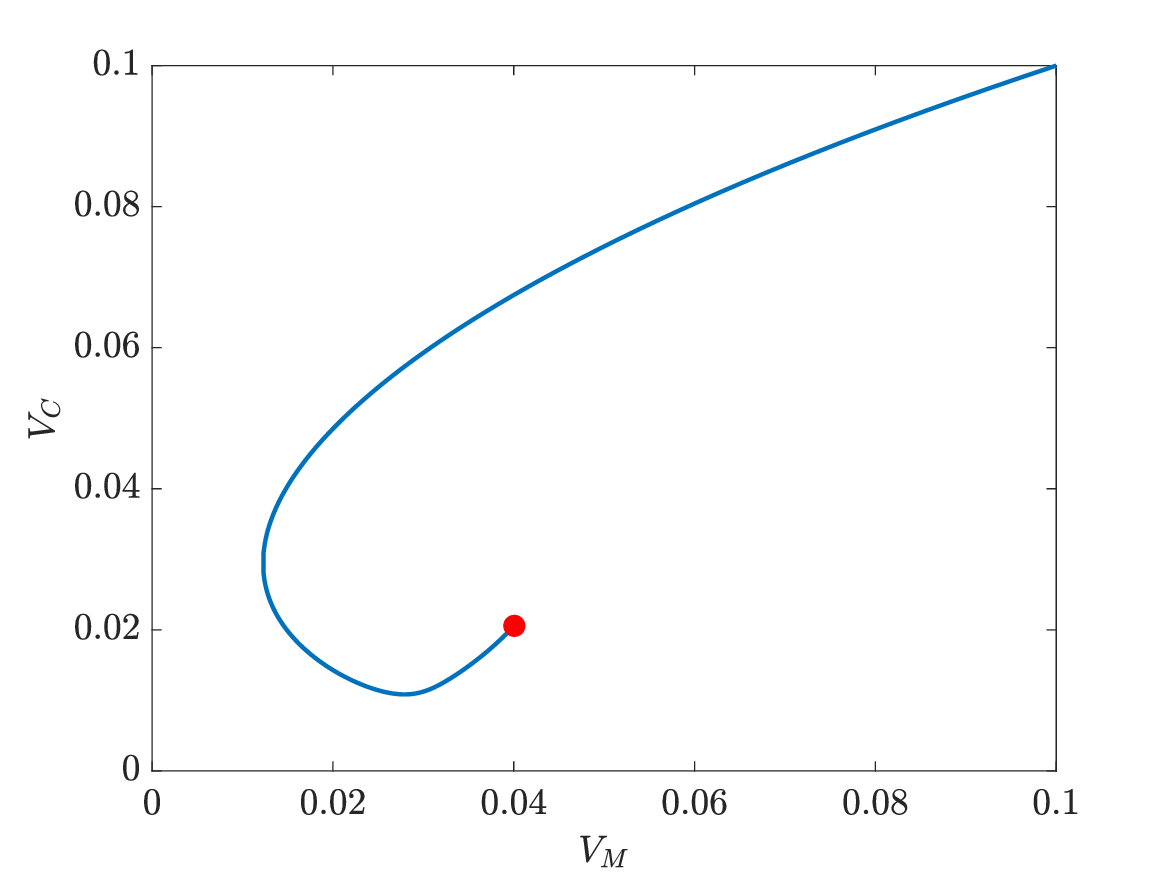}
\caption{{\bf Evolution of key observable macroscopic quantities for the mean-field model.} {\bf Top row.} Plots of the components of the numerical solution of the system~\eqref{eq:mean_K_NDMC} (left panel) and the system~\eqref{eq:var_MF_NDMC} (right panel) for $t \in [0,100]$. {\bf Center and bottom rows.} Corresponding trajectories of the numerical solutions in the phase planes $(m_N, m_D)$ (center, left panel), $(m_M, m_C)$ (center, right panel), $(V_N, V_D)$ (bottom, left panel), and $(V_M, V_C)$ (bottom, right panel). The red dots highlight the points corresponding to the components of the equilibria~\eqref{eq:MNDMCinfty} and \eqref{eq:VNDMCinfty}. Numerical simulations were carried out under initial conditions of components $m_N(0) = 9 $, $m_D(0)= 1$, $m_M(0) = 0.1$, $m_C(0) = 0.5$, and $V_J(0) = 0.1$ for all $J \in \{N,D,M,C\}$, and the parameter values $\beta_N = 0.2$, $\beta_D = 0.1$, $\beta_M = 0.2$, $\beta_C = 0.1$, $\gamma_M = \gamma_C = 0.05$, and $\sigma_J^2 = 0.01$  for all $J \in \{N,D,M,C\}$, which were chosen with exploratory aim and are to be intended as being dimensionless. }
    \label{fig:fig1}
\end{figure}

\section{Long-time asymptotics for the mean-field model}
\label{sect:3}
In this section, we study long-time asymptotics for the mean-field model~\eqref{eq:MFN}-\eqref{eq:MFoBCs}. In more detail, we first characterise the components of the quasi-equilibrium solution and show that they converge to equilibrium as $t \to \infty$ (see Section~\ref{sect:31}). Then we prove the convergence, with respect to an appropriate norm, of the solution to the quasi-equilibrium solution as $t \to \infty$, which implies, in the light of the convergence to equilibrium of the quasi-equilibrium solution as $t \to \infty$, the long-time convergence of the solution of the mean-field model to equilibrium (see Section~\ref{sect:32}).

\subsection{Characterisation of the quasi-equilibrium distribution}\label{sect:31}
The components of the quasi-equilibrium distribution, $f_J^q(x,t)$ for $J \in \{N,D,M,C\}$, of the mean-field model \eqref{eq:MFN}-\eqref{eq:MFoBCs} are the components of the solution to the following system of differential equations
\begin{equation}
\label{eq:fqNDMC}
\begin{split}
&(\beta_N m_C(t) x - \beta_D m_M(t) m_D(t))f_N^q(x,t) + \dfrac{\sigma^2_N m_C(t)}{2} \dfrac{\partial}{\partial x}(x^2 f_N^q(x,t)) = 0, \\
&(\beta_D m_M(t) x - \beta_N m_N(t) m_C(t))f_D^q(x,t) + \dfrac{\sigma^2_D m_M(t)}{2} \dfrac{\partial}{\partial x}(x^2 f_D^q(x,t)) = 0, \\
&(\beta_M x - \gamma_M  m_D(t))f_M^q(x,t) + \dfrac{\sigma^2_M}{2} \dfrac{\partial}{\partial x}(x^2 f_M^q(x,t)) = 0, \\
&(\beta_C x - \gamma_C  m_M(t))f_C^q(x,t) + \dfrac{\sigma^2_C}{2} \dfrac{\partial}{\partial x}(x^2 f_C^q(x,t)) = 0,
\end{split}
\quad x \in \mathbb{R}_+
\end{equation}
subject to the boundary conditions~\eqref{eq:MFNBCs} and~\eqref{eq:MFoBCs}.

Solving the system of differential equations~\eqref{eq:fqNDMC} subject to the boundary conditions~\eqref{eq:MFNBCs} and~\eqref{eq:MFoBCs}, we find that the non-negative components, with unit mass, of the quasi-equilibrium distribution are probability density functions of inverse Gamma distributions with constant shape parameters, $\nu_J$, and time-dependent scale parameters, $\omega_J(t)$. Specifically, for the $N$ and $D$ components we have
\begin{equation}
\label{eq:fqND}
f_N^q(x,t) = \dfrac{\omega_N^{\nu_N}(t)}{\Gamma(\nu_N)}\left(\dfrac{1}{x} \right)^{\nu_N+1} e^{-\omega_N(t)/x}, \qquad 
f_D^q(x,t) = \dfrac{\omega_D^{\nu_D}(t)}{\Gamma(\nu_D)}\left(\dfrac{1}{x} \right)^{\nu_D+1} e^{-\omega_D(t)/x},
\end{equation}
where $\Gamma$ is the Gamma function and
\begin{equation}
\label{eq:nuND}
\begin{split}
\nu_N &= 1+\dfrac{2\beta_N}{\sigma^2_N},\qquad \omega_N(t) = \dfrac{2\beta_D m_M(t) m_D(t)}{\sigma_N^2 m_C(t)}, \\
\nu_D &= 1+\dfrac{2\beta_D}{\sigma^2_D },\qquad \omega_D(t) = \dfrac{2\beta_N m_N(t) m_C(t)}{\sigma_D^2 m_M(t)}, 
\end{split}
\end{equation}
while for the $M$ and $C$ components we have
\begin{equation}
\label{eq:fqMC}
\begin{split}
f_M^q(x,t) = \dfrac{\omega_M^{\nu_M}(t)}{\Gamma(\nu_M)}\left(\dfrac{1}{x} \right)^{\nu_M+1} e^{-\omega_M(t)/x},\qquad
f_C^q(x,t) = \dfrac{\omega_C^{\nu_C}(t)}{\Gamma(\nu_C)}\left(\dfrac{1}{x} \right)^{\nu_C+1} e^{-\omega_C(t)/x},
\end{split}
\end{equation}
with 
\begin{equation}
\label{eq:nuMC}
\begin{split}
\nu_M &= 1+\dfrac{2\beta_M}{\sigma^2_M},\qquad \omega_M(t) = \dfrac{2\gamma_M}{\sigma_M^2} m_D(t), \\
\nu_C &= 1+\dfrac{2\beta_C}{\sigma^2_C },\qquad \omega_C(t) = \dfrac{2\gamma_C}{\sigma_C^2 } m_M(t). 
\end{split}
\end{equation}
Since the shape parameters defined via~\eqref{eq:nuND} and~\eqref{eq:nuMC} are such that $\nu_J>1$ for all $J \in \{N,D,M,C\}$, through the formula for the mean of the inverse Gamma distribution we find 
\[
m_N^q(t) = \dfrac{\omega_N(t)}{\nu_N-1} = \dfrac{\beta_D m_M(t) m_D(t)}{\beta_N m_C(t)}
\qquad m_D^q(t) =\dfrac{\omega_D(t)}{\nu_D-1} = \dfrac{\beta_N  m_C(t) m_N(t)}{\beta_D m_M(t)}
\]
and
\[
m_M^q(t) =  \dfrac{\omega_M(t)}{\nu_M-1} = \dfrac{\gamma_M m_D(t)}{\beta_M}, \qquad
m_C^q(t) =  \dfrac{\omega_C(t)}{\nu_C-1} =\dfrac{\gamma_C m_M(t)}{\beta_C}.
\]
%

Considering non-negative initial data that satisfy conditions~\eqref{eq:boundedmeanIC} along with condition~\eqref{eq:ICNpmD}, which gives the constraint corresponding to~\eqref{eq:NpmD}, the above expressions for $\omega_J(t)$, $J \in \{N,D,M,C\}$, along with the bounds~\eqref{eq:boundedmean} on $m_J(t)$, $J \in \{N,D,M,C\}$, imply that there exists some constants $0<c_J<C_J<\infty$ such that
\begin{equation}
\label{eq:boundsomegaJ}
c_J \le \omega_J(t) \le C_J \quad \forall t \ge 0, \quad J \in \{N,D,M,C\},
\end{equation}
while the above expressions for $m^q_J(t)$ imply that 
\begin{equation}
\label{eq:convmJ}
m_J^q(t) \to m_J^\infty \quad \text{as } t \to \infty, \quad J \in \{N,D,M,C\},
\end{equation}
with $m_J^\infty$ given by~\eqref{eq:MNDMCinfty}.  Furthermore, the expressions of $\omega_J(t)$, $J \in \{N,D,M,C\}$, ensure that there exist some constants $0\le C_J^p <\infty$ such that, denoting the derivative of $\omega_J(t)$ by $\omega_J'(t)$, 
\begin{equation}
\label{eq:boundsomegaJp}
 \omega_J^\prime(t) \le C_J^p, \qquad \forall t \ge 0, \quad J \in \{N,D,M,C\}, 
\end{equation}
since $m_C(t)>0$ and $m_M(t)>0$ for all $t\ge0$ (cf. \eqref{eq:boundedmean}). 

Moreover, under the additional assumptions~\eqref{eq:asssigmabetaJ} on the model parameters, which ensure that the shape parameters defined via~\eqref{eq:nuND} and~\eqref{eq:nuMC} are also such that $\nu_J>2$ for all $J \in \{N,D,M,C\}$, the variances of the inverse Gamma quasi-equilibrium distributions are given by
\[
V_N^q(t) = \dfrac{\beta_D^2 (m_M(t))^2 (m_D(t))^2 \sigma_N^2}{\beta_N^2 (m_C(t))^2(2\beta_N - \sigma^2_N)},\qquad
V_D^q(t) = \dfrac{\beta_N^2 (m_N(t))^2 (m_C(t))^2 \sigma_D^2}{\beta^2_D (m_M(t))^2(2\beta_D  -  \sigma_D^2)}
\]
and
\[
V_M^q(t) =\dfrac{\gamma_M^2 \sigma^2_M}{\beta_M^2(2\beta_M-\sigma^2_M)} (m_D(t))^2 , \qquad V_C^q(t) = \dfrac{\gamma_C^2 (m_M(t))^2 \sigma^2_C}{\beta_C^2(2\beta_C - \sigma^2_C)}.
\]
The above expressions for $V^q_J(t)$ along with the asymptotic results~\eqref{eq:convmJ} allow us to conclude that 
\begin{equation}
\label{eq:convVJ}
V_J^q(t) \to V_J^\infty \quad \text{as } t \to \infty, \quad J \in \{N,D,M,C\},
\end{equation}
with $V_J^\infty$ given by~\eqref{eq:VNDMCinfty}.

Finally, combining the results~\eqref{eq:convmJ} with the expressions~\eqref{eq:fqND} and~\eqref{eq:fqMC} for $f_J^q(x,t)$, $J \in \{N,D,M,C\}$, yields 
\begin{equation}
\label{eq:convfJ}
\|f_J^q(t) - f_J^\infty \|_{L^{\infty}(\mathbb{R}_+)} \to 0 \quad \text{as } t \to \infty, \quad J \in \{N,D,M,C\},
\end{equation}
with
\begin{equation}
\label{eq:finftyJ}
f_J^{\infty}(x) = \dfrac{(\omega^{\infty}_J)^{\nu_J}}{\Gamma(\nu_J)}\left(\dfrac{1}{x} \right)^{\nu_J+1} e^{-\omega^{\infty}_J/x},
\end{equation}
where $\nu_J = 1+\dfrac{2\beta_J}{\sigma^2_J}$ and
\begin{equation}
\label{eq:omegaJ}
\begin{split}
 \omega^{\infty}_N = \dfrac{2\beta_D \beta_C}{\sigma_N^2 \gamma_C} m^{\infty}_D, \qquad \omega^{\infty}_D = \dfrac{2\beta_N \gamma_C}{\sigma_D^2 \beta_C} m^{\infty}_N, \qquad \omega_M^{\infty} = \dfrac{2\gamma_M}{\sigma_M^2 } m^{\infty}_D, \qquad \omega^{\infty}_C =  \dfrac{2\gamma_C \gamma_M}{\sigma_C^2 \beta_M}  m^\infty_D.
\end{split}
\end{equation}

\subsection{Long-time convergence to the equilibrium distribution}\label{sect:32}
The convergence of the components, $f_J(x,t)$ for $J \in \{N,D,M,C\}$, of the solution of the mean-field model~\eqref{eq:MFN}-\eqref{eq:MFoBCs} to the corresponding quasi-equilibrium distribution functions, $f_J^q(x,t)$ for $J \in \{N,D,M,C\}$, defined via~\eqref{eq:fqND} and~\eqref{eq:fqMC}, can be obtained in the $\dot H_{-p}$ norm\footnote{The metric induced by this norm, introduced in~\cite{BG}, is closely related to the Fourier-based metrics introduced in~\cite{GTW} to study the trends to equilibrium for homogeneous kinetic equations; see also~\cite{TT} for further applications of such metrics.} defined as
\begin{equation}\label{def:H-pnorm}
\|f \|_{\dot H_{-p}}^2 = \int_{\mathbb R}|\xi|^{-2p}|\hat f(\xi)|^2 d\xi, \qquad p>\frac{1}{2},
\end{equation}
where $\hat f$ is the Fourier transform of the probability density function $f(x)$, that is, $\hat f(\xi) = \int_{\mathbb R_+}f(x)e^{-i\xi x}dx$. This is shown by Theorem~\ref{th:convergencetoeq}.

\begin{theorem}
\label{th:convergencetoeq}
Let $f_J(0,x)$, $J \in \{N,D,M,C\}$, be probability density functions on $\mathbb R_+$ such that conditions~\eqref{eq:unitmassIC}, \eqref{eq:boundedmeanIC}, \eqref{eq:ICNpmD}, and~\eqref{eq:boundedvarIC} hold and such that $\| f_J(0,x) - f_J^\infty(x)\|_{\dot H_{-p}}<\infty$ for $\frac{1}{2}<p<1$, with $f_J^\infty(x)$ given by~\eqref{eq:finftyJ}. Then, the solution to the mean-field model \eqref{eq:MFN}-\eqref{eq:MFoBCs} is such that 
\[
\sum_{J \in \{N,D,M,C\}} \| f_J - f_J^\infty\|_{\dot{H}_{-p}} \to 0 \quad \text{as } t \to \infty, \qquad \frac{1}{2}<p<1.
\]
\end{theorem}

\begin{proof} $\phantom{.}$
\\
{\bf Step 1: $f_N \to f_N^q$ in $\dot H_{-p}$, with $\frac{1}{2}<p<1$, as $t \to \infty$.} Since the right-hand side of the Fokker-Planck equation~\eqref{eq:MFN} for $f_N(x,t)$ vanishes when $f_N = f_N^q$, with $f^q_N$ being the quasi-equilibrium distribution given by \eqref{eq:fqND}, we have
\[
\begin{split}
&\dfrac{\partial}{\partial t}(f_N - f_N^q)= \\
&\quad -\dfrac{\partial f_N^q}{\partial t} + \dfrac{\sigma_N^2m_C(t)}{2}\dfrac{\partial^2}{\partial x^2}[x^2 (f_N - f_N^q)] + \dfrac{\partial}{\partial x}\left[(\beta_N m_C x - \beta_D m_M m_D)(f_N-f_N^q)\right],
\end{split}\]
which, taking the Fourier transforms, gives
\begin{equation}
\label{eq:fourierN}
\begin{split}
&\dfrac{\partial}{\partial t}(\hat f_N - \widehat{f_N^q}) = \\
&\qquad  -\dfrac{\partial \widehat{f_N^q}}{\partial t}  + \dfrac{\sigma_N^2 m_C}{2}\xi^2 \dfrac{\partial^2}{\partial \xi^2}(\hat f_N  - \widehat{f_N^q})  -  \beta_N m_C \xi \dfrac{\partial}{\partial \xi}(\hat f_N - \widehat{f_N^q}) - i\beta_D m_M m_D\xi (\hat f_N-\widehat{f_N^q}).
\end{split}
\end{equation}
Through a method analogous to the one employed in~\cite{MTZ,TT}, from~\eqref{eq:fourierN} we find
\begin{equation}
\label{eq:estim_1}
\begin{split}
\dfrac{d}{dt}\int_{\mathbb R}|\xi|^{-2p}|\hat f_N - \widehat{f_N^q}|^2d\xi \le& \int_{\mathbb R}|\xi|^{-2p} \left| \dfrac{\partial {\widehat{f_N^q}}}{\partial t}\right| |\hat f_N - \widehat{f_N^q}|d\xi\\
&-(2p-1)m_C\left[ \sigma_N^2 \dfrac{3-2p}{4} + \beta_N\right]\int_{\mathbb R} |\xi|^{-2p}|\hat f_N - \widehat{f_N^q}|^2d\xi.
\end{split}
\end{equation}
Being $f_N^q$ an inverse Gamma distribution, the following estimate holds
\[
\int_{\mathbb R_+}\left| \dfrac{\partial f_N^q}{\partial t} \right| dx\le |\omega_N'(t)|B(t), \qquad B(t) = \dfrac{2\nu_N}{\omega_N(t)},
\]
where $\omega_N'(t)$ is the derivative of $\omega_N(t)$. Moreover, since $\omega_N(t)$ is bounded (cf. \eqref{eq:boundsomegaJ}), there exists a constant $D_N>0$ such that $B(t)\le D_N$ for all $t > 0$. Hence,
\[
\left| \dfrac{\partial \widehat{f_N^q}}{\partial t}(\xi,t)\right| = \left| \int_{\mathbb R_+} \dfrac{\partial f_N^q(x,t)}{\partial t} e^{-i\xi x} dx \right| \le    \int_{\mathbb R_+} \left|\dfrac{\partial f_N^q(x,t)}{\partial t} \right| dx \le |\omega'_N(t)| D_N \quad \forall t > 0.
\]
Since both $f_N$ and $f_N^q$ are probability density functions on $\mathbb{R}_+$ we have that $\hat f_N(0,t) = \widehat{f_N^q}(0,t) = 1$. Thus, expanding in Taylor series about $\xi = 0$ we obtain
\[
|\hat f_N - \widehat{f_N^q}| =\left| \dfrac{\partial\hat{f}_N}{\partial \xi}(\xi_*,t) - \dfrac{\partial\widehat{ f_N^q}}{\partial \xi}(\xi^*,t) \right| |\xi| \le |m_N(t)  + m_N^q(t)||\xi| \le M_N |\xi|, 
\]
with $M_N = \displaystyle{\max_{t>0}\{m_N(t) + m_N^q(t)\}}<\infty$. In fact, being $f_N$ a probability density function of finite mean value, we have
\[
\left| \dfrac{\partial \hat{f}_N}{\partial \xi}\right| = \left| \int_{\mathbb R_+} \dfrac{\partial f_N}{\partial x} e^{-i\xi x}dx \right| = \left| i\int_{\mathbb R_+}xf_N(x,t) e^{-i\xi x}dx \right|\le \left| \int_{\mathbb R_+}x f_N(x,t)dx \right |.
\] 
Therefore, for any $\xi \in \mathbb R$ we have
\[
|\xi|^{-2p} \left| \dfrac{\partial \widehat{f_N^q}}{\partial t} \right| |\hat f_N - \widehat{f_N^q}| \le M_N D_N |\omega_N'| |\xi|^{1-2p}
\]
and, for any $R>0$, the following estimate on the first term on the right-hand side of \eqref{eq:estim_1} holds
\[
\begin{split}
&\int_{\mathbb R}|\xi|^{-2p} \left| \dfrac{\partial \widehat{f_N^q}}{\partial t}\right| |\hat f_N - \widehat{f_N^q}|d\xi \le \\
&\qquad \int_{|\xi|\le R}  M_N |\omega_N'|B(t) |\xi|^{1-2p} d\xi + \int_{|\xi|> R}  |\xi|^{-2p} \left| \dfrac{\partial \widehat{f_N^q}}{\partial t}\right| |\hat f_N - \widehat{f_N^q}|d\xi = \\
&\qquad M_N |\omega_N'|B(t) \dfrac{R^{2-2p}}{1-p} + \int_{|\xi|> R}  |\xi|^{-2p} \left| \dfrac{\partial \widehat{f_N^q}}{\partial t} \right| |\hat f_N - \widehat{f_N^q}|d\xi . 
\end{split}\]
Furthermore, for any $p>1/2$, the last term in the above estimate can be further estimated as
\[
\begin{split}
&\int_{|\xi|> R}  |\xi|^{-2p} \left| \dfrac{\partial \widehat{f_N^q}}{\partial t}(\xi,t) \right| |\hat f_N - \widehat{f_N^q}|d\xi \le \\ &\left( \int_{|\xi|>R} |\xi|^{-2p}  \left|  \dfrac{\partial \widehat{f_N^q}}{\partial t}\right|^2 d\xi \right)^{1/2} \left( \int_{|\xi|>R} |\xi|^{-2p}|\hat f_N - \widehat{f_N^q}|^2d\xi \right)^{1/2} \le \\
& M_N D_N |\omega_N'|\left( \int_{|\xi|>R} |\xi|^{-2p} d\xi \right)^{1/2} \left( \int_{\mathbb R} |\xi|^{-2p}|\hat f_N - \widehat{f_N^q}|^2d\xi \right)^{1/2} =   \\
& M_N D_N|\omega_N^\prime | R^{1/2 - p} \left( \dfrac{2}{2p-1} \int_{\mathbb R} |\xi|^{-2p}|\hat f_N - \widehat{f_N^q}|^2d\xi \right)^{1/2} ,
\end{split}\]
and then, for any $1/2<p<1$, we have the following estimate on the first term on the right-hand side of~\eqref{eq:estim_1}
\[
\begin{split}
&\int_{\mathbb R}|\xi|^{-2p} \left| \dfrac{\partial \widehat{f_N^q}}{\partial t}(\xi,t) \right| |\hat f_N - \widehat{f_N^q}|d\xi\le \\
&\qquad M_N D_N |\omega_N'| \dfrac{R^{2-2p}}{1-p}  + M_N D_N|\omega_N^\prime | R^{1/2 - p} \left( \dfrac{2}{2p-1} \int_{\mathbb R} |\xi|^{-2p}|\hat f_N - \widehat{f_N^q}|^2d\xi \right)^{1/2}.
\end{split}
\]
As for the optimal value of $R>0$, we find 
\[
R^* = (1-p)  \left(\frac{2}{2p-1}\int_{\mathbb R}|\xi|^{-2p}|\hat f_N - \widehat{f_N^q}|^2 d\xi\right)^{1/2}>0
\]
for any $1/2<p<1$. Therefore, the inequality can be optimised for $R = R^*$ as follows
\[
\begin{split}
&\int_{\mathbb R}|\xi|^{-2p} \left| \dfrac{\partial \widehat{f_N^q}}{\partial t}(\xi,t) \right| |\hat f_N - \widehat{f_N^q}|d\xi \le  C_pM_N D_N|\omega_N'|\left(\int_{\mathbb R}|\xi|^{-2p}|\hat f_N -\widehat{f_N^q}|^2 d\xi \right)^{\frac{2-2p}{3-2p}},
\end{split}
\] 
with 
\[
C_p = \left[ \dfrac{2}{2p-1}\right]^{\frac{(2-2p)}{3-2p}} \left[ \dfrac{1}{1-p}\right]^{\frac{-1+2p}{3-2p}}  \left[ \left( \dfrac{p-\frac{1}{2}}{2-2p} \right)^{\frac{2(2-2p)}{3-2p}} + \left( \dfrac{p-\frac{1}{2}}{2-2p}\right)^{\frac{1-2p}{3-2p}}\right]>0.
\]
Setting 
\[
z_N(t) = \int_{\mathbb R}|\xi|^{-2p}|\hat f_N - \widehat{f_N^q}|^2d\xi,
\]
from \eqref{eq:estim_1} we then obtain
\[
\dfrac{dz_N(t)}{dt} \le - (2p-1)m_C\left(\sigma_N^2 \dfrac{3-2p}{4} + \beta_N\right)z_N(t)  + C_p M_N D_N |\omega_N'| z_N(t)^{\frac{2-2p}{3-2p}},
\]
from which the change of variable $y_N = z_N^{1/(3-2p)}$ yields
\[
\dfrac{dy_N(t)}{dt} \le - \dfrac{2p-1}{3-2p}m_C\left(\sigma_N^2 \dfrac{3-2p}{4} + \beta_N\right) y_N(t) + \dfrac{C_p M_N D_N |\omega_N'|}{3-2p},
\]
that is,
\[
y_N(t) \leq \left[ y_N(0) + \int_0^t b_N(s) \exp\left\{\int_0^s a_N(\tau)d\tau)\right\}ds\right]\exp\left\{-\int_0^t a_N(s)ds\right\},
\]
with 
\[
a_N(t) = \dfrac{2p-1}{3-2p}\left(\sigma_N^2 \dfrac{3-2p}{4} + \beta_N\right)m_C(t), \qquad b_N(t) = \dfrac{C_p M_N D_N |\omega_N'(t)|}{3-2p}. 
\]
Since $\omega_N'(t)$ is bounded (cf.~\eqref{eq:boundsomegaJp}) then $b_N(t)$ is bounded as well. Moreover, since $m_C(t)$ is non-negative and converges to a positive value as  $t \to \infty$ (cf. \eqref{eq:boundedmean} and~\eqref{eq:MNDMCinfty}), the same holds for $a_N(t)$, and thus $y_N(t) \to 0$ as $t\to \infty$, which in turn ensures that $z_N(t) \to 0$ as $t\to \infty$ as well. As a result, 
\begin{equation}
\label{eq:convfN}
\| f_N - f_N^q\|_{\dot{H}_{-p}} = \int_{\mathbb R}|\xi|^{-2p}|\hat f_N - \widehat{f_N^q}|^2 d\xi  \to 0 \quad \text{as } t \to \infty, \qquad \frac{1}{2}<p<1.
\end{equation}
\\

{\bf Step 2: $f_D \to f_D^q$ in $\dot H_{-p}$, with $\frac{1}{2}<p<1$, as $t \to \infty$.} Since the right-hand side of the Fokker-Planck equation~\eqref{eq:MFo} for $f_D(x,t)$ vanishes when $f_D = f_D^q$, with $f^q_D$ being the quasi-equilibrium distribution given by \eqref{eq:fqND}, we have 
\[
\begin{split}
&\dfrac{\partial}{\partial t} (f_D - f_D^q) = \\
&\quad - \frac{\partial f_D^q}{\partial t} + \dfrac{\sigma_D^2 m_M}{2}\dfrac{\partial^2}{\partial x^2} [x^2 (f_D - f_D^q)] + \dfrac{\partial}{\partial x} \left[ (\beta_D m_M x - \beta_N m_N m_C) (f_D - f_D^q)\right].
\end{split}\]
Taking the Fourier transforms gives
\begin{equation}
\label{eq:fourierD}
\begin{split}
&\dfrac{\partial}{\partial t} (\hat{f}_D - \widehat{f_D^q}) = \\
&\qquad -\dfrac{\partial \widehat{f_D^q}}{\partial t} + \dfrac{\sigma_D^2 m_M}{2}\xi^2\dfrac{\partial^2}{\partial \xi^2} (\hat{f}_D - \widehat{f_D^q}) - \beta_D m_M \xi \dfrac{\partial}{\partial \xi}(\hat f_D - \widehat{f_D^q})  - i \beta_Nm_N m_C \xi(\hat f_D - \widehat{f_D^q}),
\end{split}
\end{equation}
and then, proceeding similarly to Step 1, exploiting the fact that $\omega_D(t)$ is bounded (cf. \eqref{eq:boundsomegaJ}) and thus there exists a constant $D_D>0$ such that $\dfrac{2\nu_D}{\omega_D(t)}\le D_D$ for all $t > 0$, and considering $\frac{1}{2}<p<1$, we find
\[\begin{split}
&\dfrac{d}{dt} \int_{\mathbb R} |\xi|^{-2p}|\hat f_D - \widehat{f_D^q}|^2 d\xi \ \\
&\qquad \le-(2p-1) m_M \left(\sigma^2_D \dfrac{3-2p}{4} + \beta_D \right)\int_{\mathbb R} |\xi|^{-2p}|\hat f_D - \widehat{f_D^q}|^2 d\xi +\\
&\qquad\qquad  C_p M_D D_D |\omega_D'| \left( \int_{\mathbb R} |\xi|^{-2p} |\hat f_D - \widehat{f_D^q}|^2d\xi\right)^{\frac{2-2p}{3-2p}}, 
\end{split}\]
where $\omega_D'(t)$ is the derivative of $\omega_D(t)$ and $M_D = \displaystyle{\max_{t>0}\{m_D(t) + m_D^q(t)\}}$. Hence, as similarly done in Step 1, setting first
$$
z_D(t) = \int_{\mathbb R} |\xi|^{-2p}|\hat f_D - \widehat{f_D^q}|^2d\xi
$$
and then $y_D = z_D^{\frac{1}{3-2p}}$, we find
\[
y_D(t) \le \left\{ y_D(0) + \int_0^t b_D(s) \exp\left[\int_0^s a_D(\tau)d\tau)\right]ds\right\}\exp\left[-\int_0^t a_D(s)ds\right],
\]
where
\[
a_D(t) = \dfrac{2p-1}{3-2p} \left( \sigma^2_D \dfrac{3-2p}{4} + \beta_D\right) m_M(t), \qquad b_D(t) = \dfrac{C_p M_D D_D |\omega_D'(t)|}{3-2p}. 
\]
Since $\omega_D'(t)$ is bounded (cf.~\eqref{eq:boundsomegaJp}) then $b_D(t)$ is bounded as well. Moreover, since $m_M(t)$ is non-negative and converges to a positive value as  $t \to \infty$ (cf. \eqref{eq:boundedmean} and~\eqref{eq:MNDMCinfty}), the same holds for $a_D(t)$, and thus $y_D(t) \to 0$ as $t\to \infty$, which in turn ensures that $z_D(t) \to 0$ as $t\to \infty$ as well. As a result, 
\begin{equation}
\label{eq:convfD}
\| f_D - f_D^q\|_{\dot{H}_{-p}} = \int_{\mathbb R}|\xi|^{-2p}|\hat f_D - \widehat{f_D^q}|^2 d\xi  \to 0 \quad \text{as } t \to \infty, \qquad \frac{1}{2}<p<1.
\end{equation}

{\bf Step 3: $f_M \to f_M^q$ and $f_C \to f_C^q$ in $\dot H_{-p}$ , with $\frac{1}{2}<p<1$, as $t \to \infty$.}   Since the right-hand side of the Fokker-Planck equation~\eqref{eq:MFo} for $f_M(x,t)$ vanishes when $f_M = f_M^q$, with $f^q_M$ being the quasi-equilibrium distribution given by \eqref{eq:fqMC}, we have 
\[
\dfrac{\partial}{\partial t} (f_M - f_M^q) = \dfrac{\partial}{\partial x} \left[ (\beta_M x - \gamma_M m_D(t))(f_M - f_M^q)\right] +\dfrac{\sigma_M^2}{2}\dfrac{\partial^2}{\partial x^2} [x^2 (f_M- f_M^q)].
\]
Then, through a method analogous to the ones employed in Steps 1 and 2, setting
\[
z_M(t) = \int_{\mathbb R} |\xi|^{-2p}|\hat f_M - \widehat{f_M^q}|^2 d\xi, 
\]
we find
\[
\dfrac{d}{dt} z_M(t) \le \int_{\mathbb R} |\xi|^{-2p} \left| \dfrac{\partial \widehat{f_M^q}}{\partial t} \right| |\hat f_M - \widehat{f_M^q}|  d\xi - (2p-1) \left[ \dfrac{\sigma_M^2(3-2p)}{4} + \beta_M\right] z_M. 
\]
Using the fact that, since $\omega_M(t)$ is bounded (cf. \eqref{eq:boundsomegaJ}) and thus there exists a constant $D_M>0$ such that $\dfrac{2\nu_M}{\omega_M(t)}\le D_M$ for all $t>0$, and considering $\frac{1}{2}<p<1$, the following estimate holds
\[
\int_{\mathbb R} |\xi|^{-2p} \left| \dfrac{\partial \widehat{f_M^q}}{\partial t}\right| |\hat f_M - \widehat{f_M^q}| d\xi \le C_p M_M D_M z_M^{\frac{2-2p}{3-2p}},
 \]
where $M_M = \displaystyle{\max_{t>0}\{m_M(t) + m_M^q(t)\}}$, from the above differential inequality we find
\[
\dfrac{d}{dt}z_M(t) \le -(2p-1)\left[ \dfrac{\sigma_M^2(3-2p)}{4} + \beta_M\right]z_M + C_p M_M D_M |\omega_M'| z_M^{\frac{2-2p}{3-2p}},
\]
where $\omega_M'(t)$ is the derivative of $\omega_M(t)$. Setting $y_M = z_M^{\frac{2-2p}{3-2p}}$, the latter differential inequality yields
\[
\dfrac{d}{dt} y_M\le -(2p-1)\left[ \dfrac{\sigma_M^2(3-2p)}{4} + \beta_M\right]y_M + \dfrac{C_p M_M D_M |\omega_M'|}{3-2p},
\]
from which, noting that $\omega_M'(t)$ is bounded (cf.~\eqref{eq:boundsomegaJp}), we conclude that $y_M(t) \to 0$ as $t\to \infty$, which in turn ensures that $z_M(t) \to 0$ as $t\to \infty$ as well. As a result, 
\begin{equation}
\label{eq:convfM}
\| f_M - f_M^q\|_{\dot{H}_{-p}} = \int_{\mathbb R}|\xi|^{-2p}|\hat f_M - \widehat{f_M^q}|^2 d\xi  \to 0 \quad \text{as } t \to \infty, \qquad \frac{1}{2}<p<1.
\end{equation}

In a similar way, introducing the notation 
$$
z_C = \displaystyle{\int_{\mathbb R}|\xi|^{-2p}|\hat f_C - \widehat{f^q_C}|^2 d\xi}
$$
and setting $y_C = z_C^{\frac{2-2p}{3-2p}}$, we find
\begin{equation}
\label{eq:Bern_C}
\dfrac{d}{dt} y_C\le -(2p-1) \left[ \dfrac{\sigma_C^2(3-2p)}{4} + \beta_C \right] y_C  + \dfrac{C_pM_C D_C |\omega_C'|}{3-2p},
\end{equation}
where $\omega_C'(t)$ is the derivative of $\omega_C(t)$ and $M_C = \displaystyle{\max_{t>0}\{m_C(t) + m_C^q(t)\}}$. From~\eqref{eq:Bern_C}, noting that $\omega_C'(t)$ is bounded (cf.~\eqref{eq:boundsomegaJp}), we conclude that $y_C(t) \to 0$ as $t\to \infty$, which in turn ensures that $z_C(t) \to 0$ as $t\to \infty$ as well. As a result, 
\begin{equation}
\label{eq:convfC}
\| f_C - f_C^q\|_{\dot{H}_{-p}} = \int_{\mathbb R}|\xi|^{-2p}|\hat f_C - \widehat{f_C^q}|^2 d\xi  \to 0 \quad \text{as } t \to \infty, \qquad \frac{1}{2}<p<1.
\end{equation}

{\bf Step 4: $f_J \to f_J^{\infty}$ in $\dot H_{-p}$, with $\frac{1}{2}<p<1$, as $t \to \infty$ for $J \in \{N,D,M,C\}$.} We start by noting that  
\[
\sum_{J \in \{N,D,M,C\}} \| f_J - f_J^\infty\|_{\dot{H}_{-p}} \le \sum_{J \in \{N,D,M,C\}} \| f_J - f_J^q\|_{\dot{H}_{-p}} + \sum_{J \in \{N,D,M,C\}} \| f^q_J - f_J^\infty\|_{\dot{H}_{-p}}.
\]
Then, using the fact that the results~\eqref{eq:convfN}, \eqref{eq:convfD}, \eqref{eq:convfM}, and~\eqref{eq:convfC} ensure that the first term on the right-hand side vanishes as $t \to \infty$ for $\frac{1}{2}<p<1$ and the result~\eqref{eq:convfJ} ensures that also the second term on the right-hand side vanishes as $t \to \infty$, we conclude that 
\[
\sum_{J \in \{N,D,M,C\}} \| f_J - f_J^\infty\|_{\dot{H}_{-p}} \to 0 \quad \text{as } t \to \infty, \qquad \frac{1}{2}<p<1.
\]
\end{proof}

\section{Numerical tests}\label{sect:4}
We now illustrate, by means of a sample of results of numerical simulations, the theoretical results obtained in the previous sections. In more detail, in Section~\ref{sec:numres:1} we numerically show that, under the time scaling~\eqref{eq:timescalingeps} and the parameter scaling~\eqref{eq:paramscalingeps}, when the scaling parameter $\epsilon$ is sufficiently small, the solution of the kinetic model~\eqref{eq:system} converges to the solution of the mean-field model~\eqref{eq:MFN}-\eqref{eq:MFoBCs}. Then, in Section~\ref{sec:numres:2} we numerically show that the components of the solution of the mean-field model converge, in the norm defined via~\eqref{def:H-pnorm}, to the equilibrium distributions defined via~\eqref{eq:finftyJ} as $t \to \infty$. A summary of the set-up of numerical simulations  is provided in Section~\ref{sec:numres:0}.

\subsection{Set-up of numerical simulations}
\label{sec:numres:0}
Numerical simulations of the kinetic model are carried out through direct Monte Carlo (DSMC) methods~\cite{PR,PT}, employing a system of $N = 10^5$ particles. Distribution functions are reconstructed through standard histograms and the corresponding moments are estimated through classical Monte Carlo estimators from the particles' system. Moreover, the system of Fokker-Planck equations of the mean-field model is solved numerically through structure-preserving methods \cite{PZ}. These methods are capable of reproducing long-time statistical properties of  exact steady-state solutions with arbitrary accuracy, while ensuring preservation of positivity of the solution along with consistent entropy dissipation. For the structure preserving scheme, we employ a uniform grid over the interval $[0,L]$ with $L=12$, which is discretised with $N_x = 801$ grid points. We use a 6th order discretization in space, with step $\Delta x$, and a 2nd order semi-implicit discretization in time with step $\Delta t = 0.5\Delta x$. 

The parameter values used to carry out numerical simulations are provided in Table~\ref{Table:params}. Such parameter values were chosen with exploratory aim (i.e. to observe the effects of all the mechanisms incorporated into the model over a reasonable computational time scale) and they are to be intended as being dimensionless. Furthermore, we define the initial cell distribution functions as the following uniform distributions 
\begin{equation}
\label{eq:init_f}
f_J(x,0) = 
\begin{cases}
 \dfrac{1}{2m_{J}(0)} & x \in \left[m_J(0) - \dfrac{1}{2}\sqrt{\dfrac{6}{5}},m_J(0) + \dfrac{1}{2}\sqrt{\dfrac{6}{5}}\right], \\
 0 & \textrm{elsewhere}, 
\end{cases}
\end{equation}
with $m_{N}(0)= 9, m_{D}(0)=1, m_{M}(0)= 0.6, m_{C}(0)= 0.6$. The distributions defined via~\eqref{eq:init_f} are such that $V_J(0) = 0.1$ for all $J \in \{N,D,M,C\}$. The initial distributions of particles for Monte Carlo simulations are sampled from~\eqref{eq:init_f}.


\begin{table}
  \centering
  \caption{Parameter values used in numerical simulations.}
  \begin{tabular}{|c|c|l|}
    \hline
    {\footnotesize \textbf{Parameter}} & {\footnotesize \textbf{Value}} & {\footnotesize \textbf{Description}}\\
    \hline\hline
    {\footnotesize$\beta_N$} & {\footnotesize 0.2 }& {\footnotesize Maximum relative change in normal cell density due to interactions with T lymphocytes} \\ 
    {\footnotesize$\beta_D$} & {\footnotesize 0.1  }& {\footnotesize Maximum relative change in damaged cell density due to interactions with macrophages} \\
    {\footnotesize$\beta_M$} & {\footnotesize 0.2  }& {\footnotesize Rate of natural decay of macrophage density} \\
    {\footnotesize$\beta_C$} & {\footnotesize 0.1 }& {\footnotesize Rate of natural decay of T lymphocyte density} \\
    {\footnotesize$\sigma_N$} & {\footnotesize 0.01}& {\footnotesize Standard deviation of zero-mean random variable modeling fluctuations in normal cell density} \\
    {\footnotesize$ \sigma_D$} & {\footnotesize 0.01 }& {\footnotesize Standard deviation of zero-mean random variable modeling fluctuations in damaged cell density} \\
    {\footnotesize$ \sigma_M$} & {\footnotesize 0.01 }& {\footnotesize Standard deviation of zero-mean random variable modeling fluctuations in macrophage density} \\
    {\footnotesize$ \sigma_C$} & {\footnotesize 0.01 }& {\footnotesize Standard deviation of zero-mean random variable modeling fluctuations in T-lymphocyte density} \\
    {\footnotesize$ \gamma_M$} & {\footnotesize 0.05 }& {\footnotesize Rate of macrophage density growth induced by damaged cells} \\
    {\footnotesize$\gamma_C$} & {\footnotesize 0.05 }& {\footnotesize Rate of T-lymphocyte density growth induced by macrophages}  \\
    \hline\hline
  \end{tabular}
  \label{Table:params}
\end{table}

\subsection{Consistency between the kinetic model and its mean-field limit}
\label{sec:numres:1}
The plots in Figures~\ref{fig:mean_evol} and~\ref{fig:var_evol} display the evolution of the means, $m_J(t)$ defined via~\eqref{eq:mJ}, and the variances, $V_J(t)$ defined via~\eqref{eq:VJ}, of the distribution functions, $f_J(x,t)$ with $J \in \{N,D,M,C\}$, of the kinetic model, under the time scaling~\eqref{eq:timescalingeps} and the parameter scaling~\eqref{eq:paramscalingeps}, for different values of the scaling parameter $\epsilon$, i.e. $\epsilon \in \{10^{-3},10^{-2},10^{-1}, 5 \times 10^{-1}\}$, as well as the means and the variances of the distribution functions of the corresponding mean-field model. As expected, based on the systems of differential equations for the means and the variances of the kinetic model and the corresponding systems for the mean-field model derived in Sections~\ref{sec:macrokin} and~\ref{sec:macromf}, changing the value of $\epsilon$ does not lead to appreciable changes in the dynamics of the means, while it does lead to more significant changes in the dynamics of the variances. Moreover, for $\epsilon$ sufficiently small, there is an excellent quantitative agreement between the dynamics not only of the means but also of the variances of the kinetic model and those of the mean-field model. This validates the formal limiting procedure employed to obtain the system of Fokker-Planck equations of the mean-field model from the kinetic model, under the time scaling~\eqref{eq:timescalingeps} and the parameter scaling~\eqref{eq:paramscalingeps},  in the limit  $\epsilon \rightarrow 0^+$.

\begin{figure}
    \centering
        \includegraphics[width=0.48\textwidth]{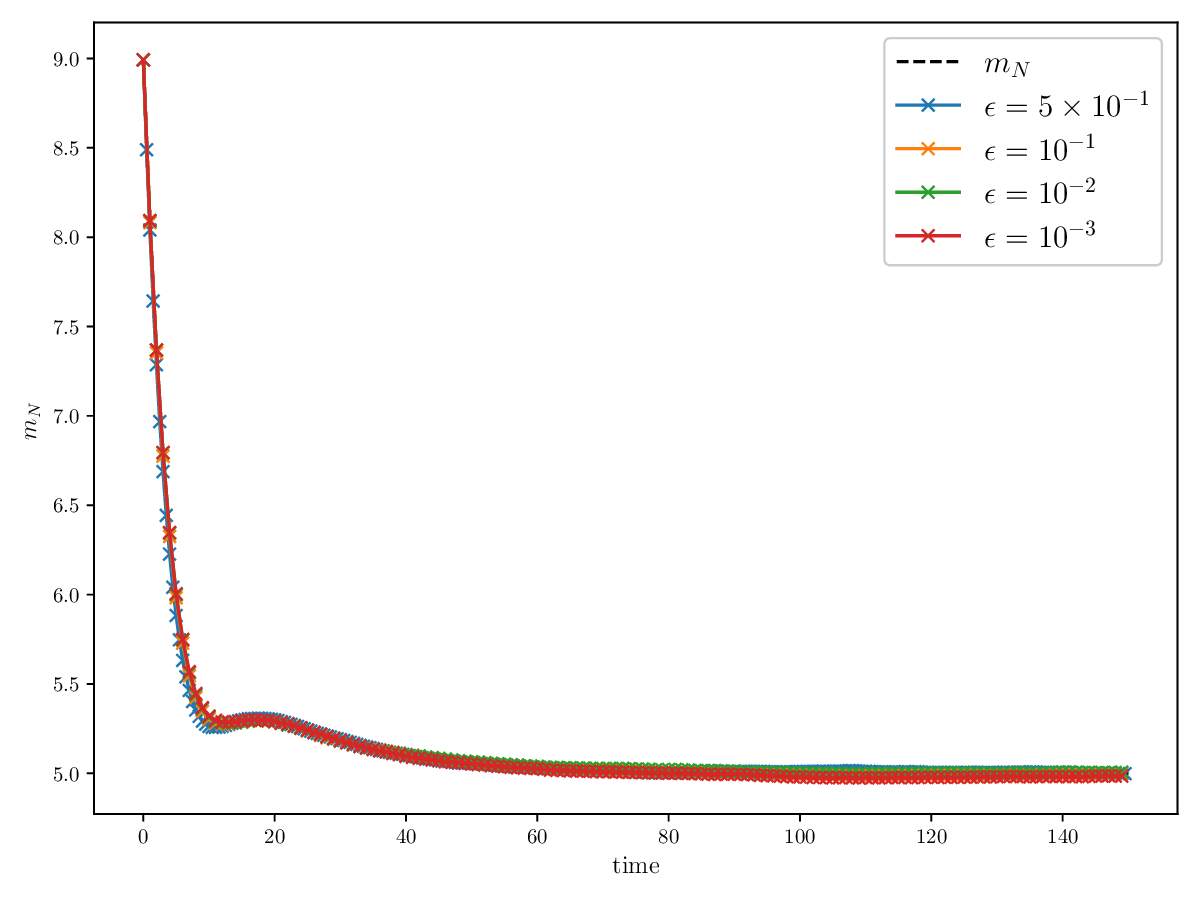}
        \includegraphics[width=0.48\textwidth]{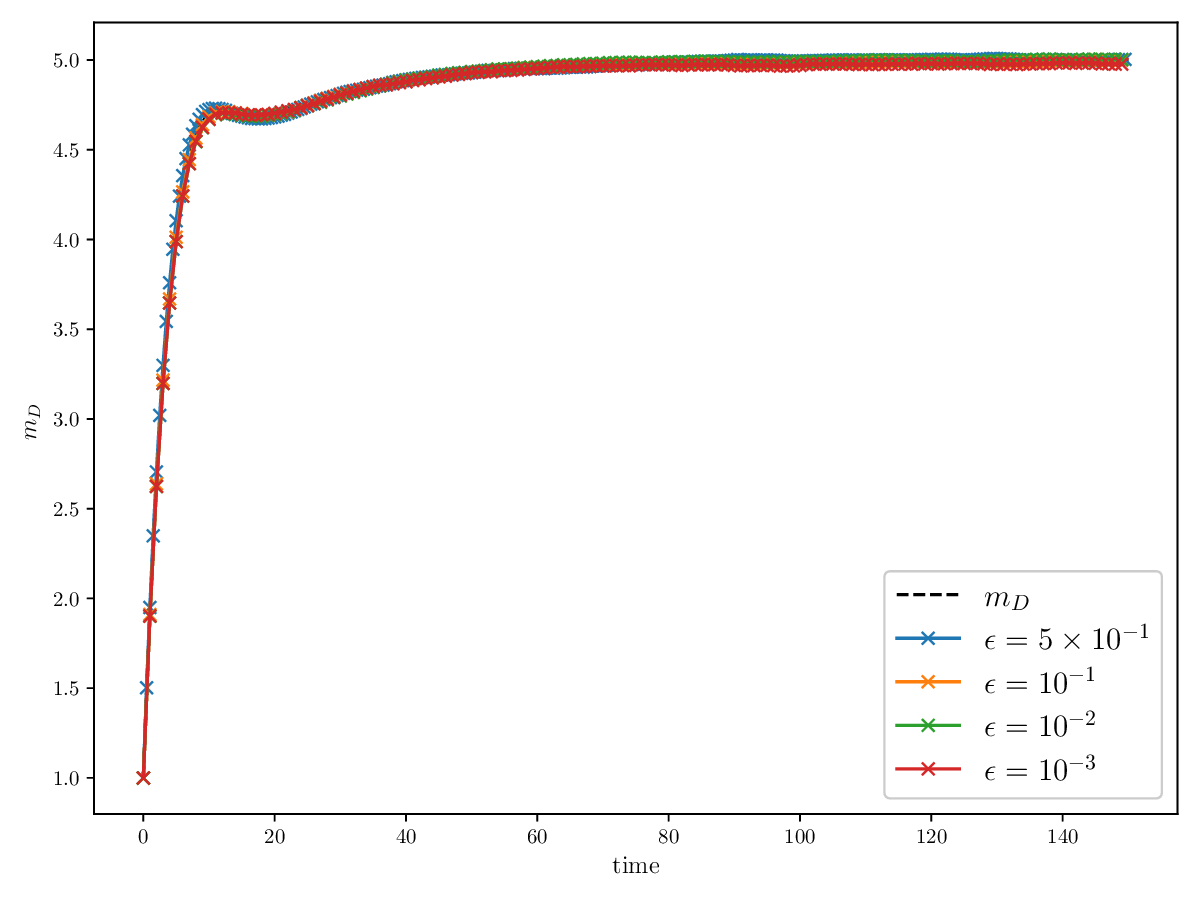}
        \includegraphics[width=0.48\textwidth]{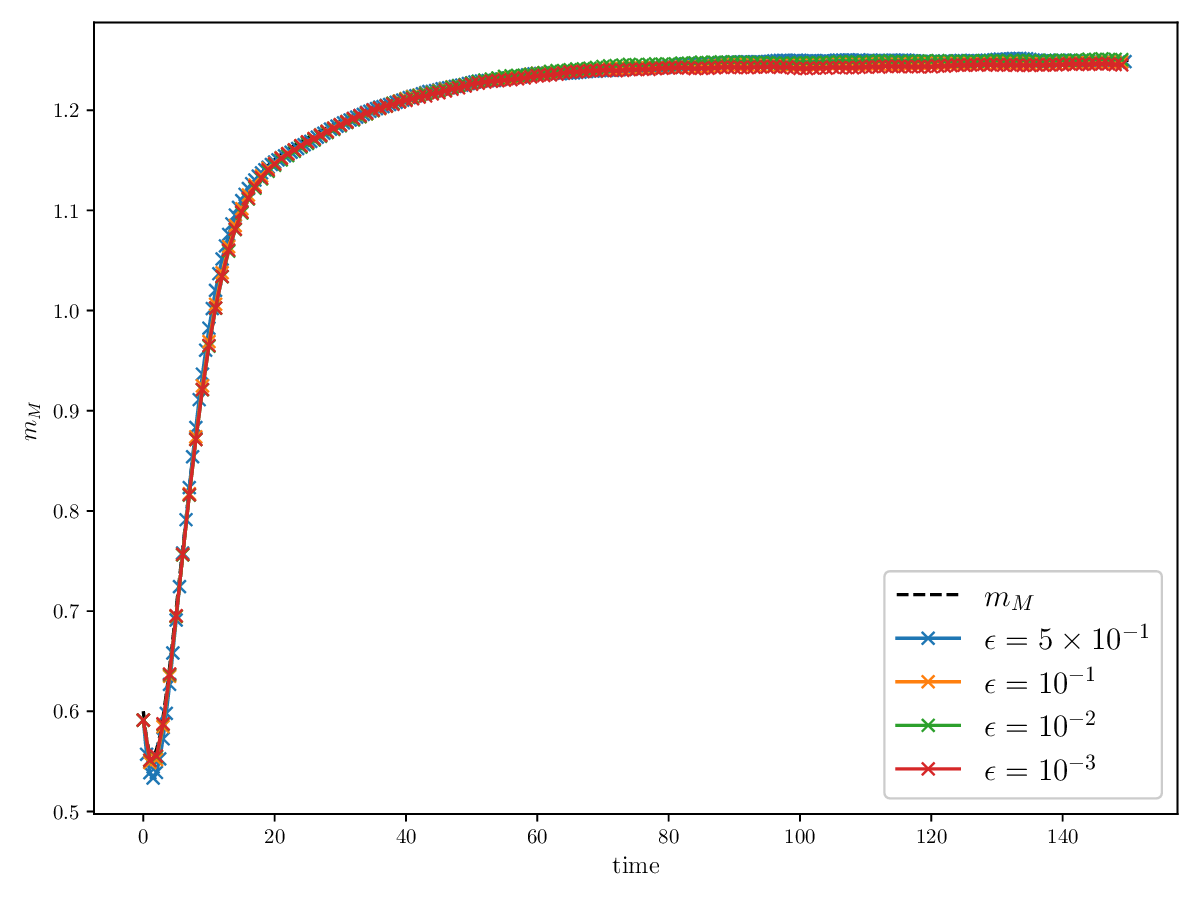}
        \includegraphics[width=0.48\textwidth]{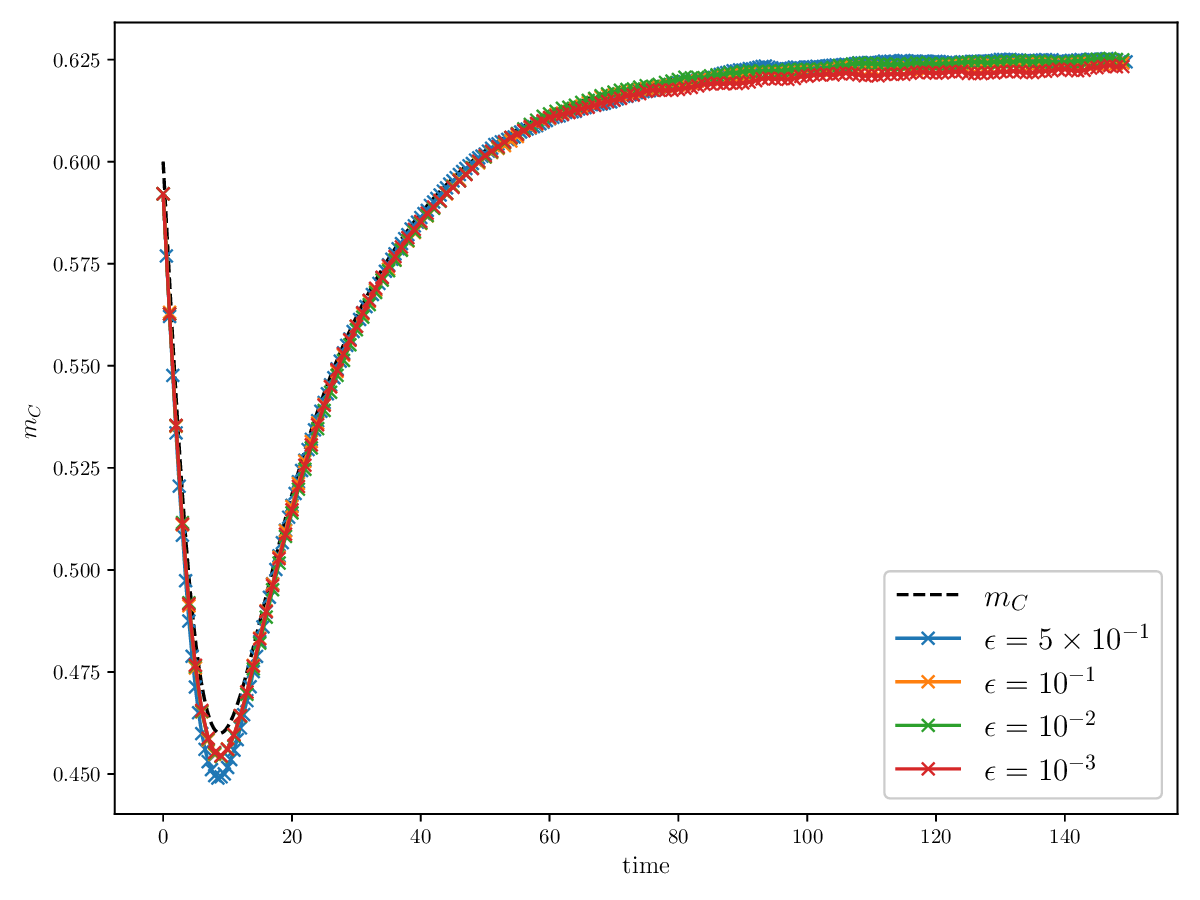}
    \caption{\textbf{Consistency between the kinetic model and its mean-field limit.} Solid, coloured lines highlight the dynamics of the means, $m_J(t)$ defined via~\eqref{eq:mJ}, of the distribution functions, $f_J(x,t)$, with $J \in \{N,D,M,C\}$, of the kinetic model~\eqref{eq:system}, under the time scaling~\eqref{eq:timescalingeps} and the parameter scaling~\eqref{eq:paramscalingeps}, for the values of the scaling parameter $\epsilon$ provided in the legends. Dashed, black lines highlight the dynamics of the means of the distribution functions of the mean-field model~\eqref{eq:MFN}-\eqref{eq:MFoBCs}. Numerical simulations were carried out under initial conditions of components defined via~\eqref{eq:init_f} and the parameter values listed in Table~\ref{Table:params}.
    }
    \label{fig:mean_evol}
\end{figure}


\begin{figure}[H]
    \centering
    \begin{subfigure}[b]{0.48\textwidth}
        \centering
        \includegraphics[width=\textwidth]{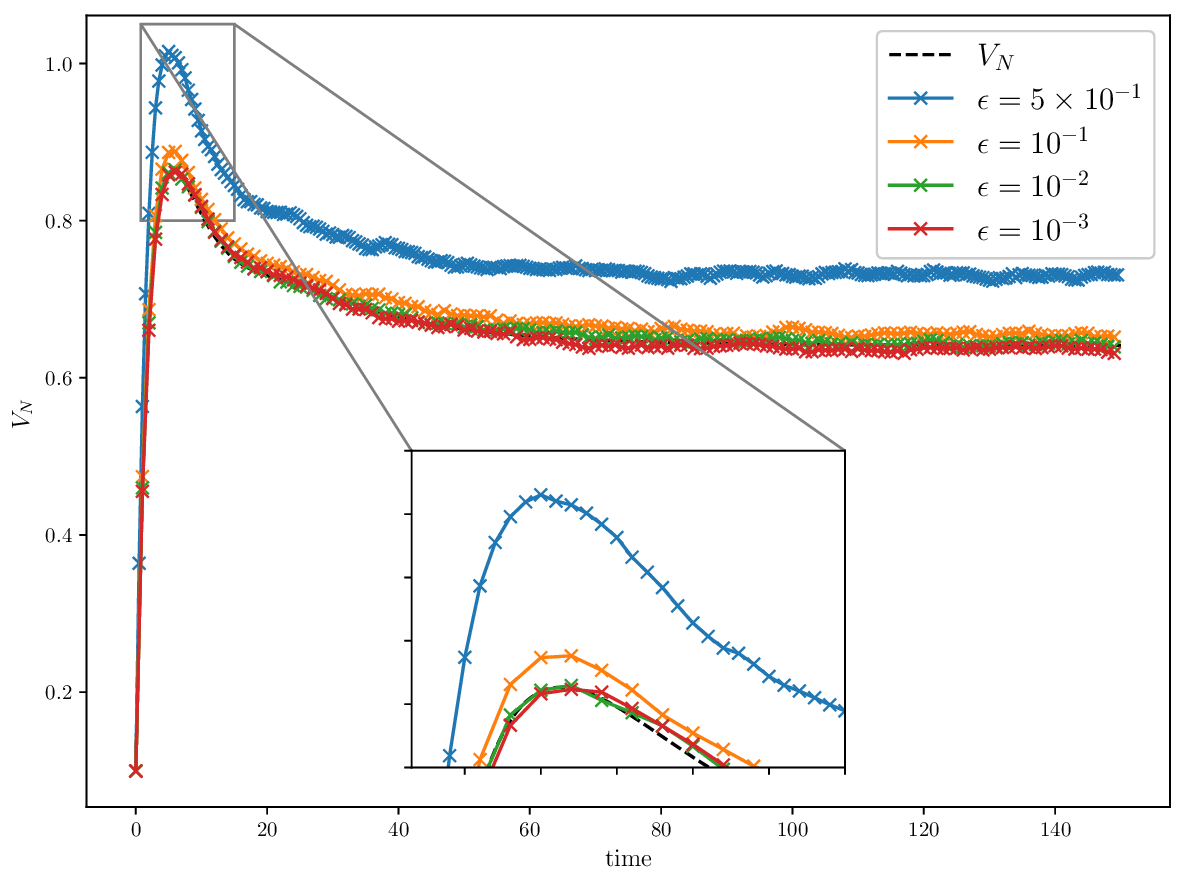}
    \end{subfigure}
    \hfill
    \begin{subfigure}[b]{0.48\textwidth}
        \centering
        \includegraphics[width=\textwidth]{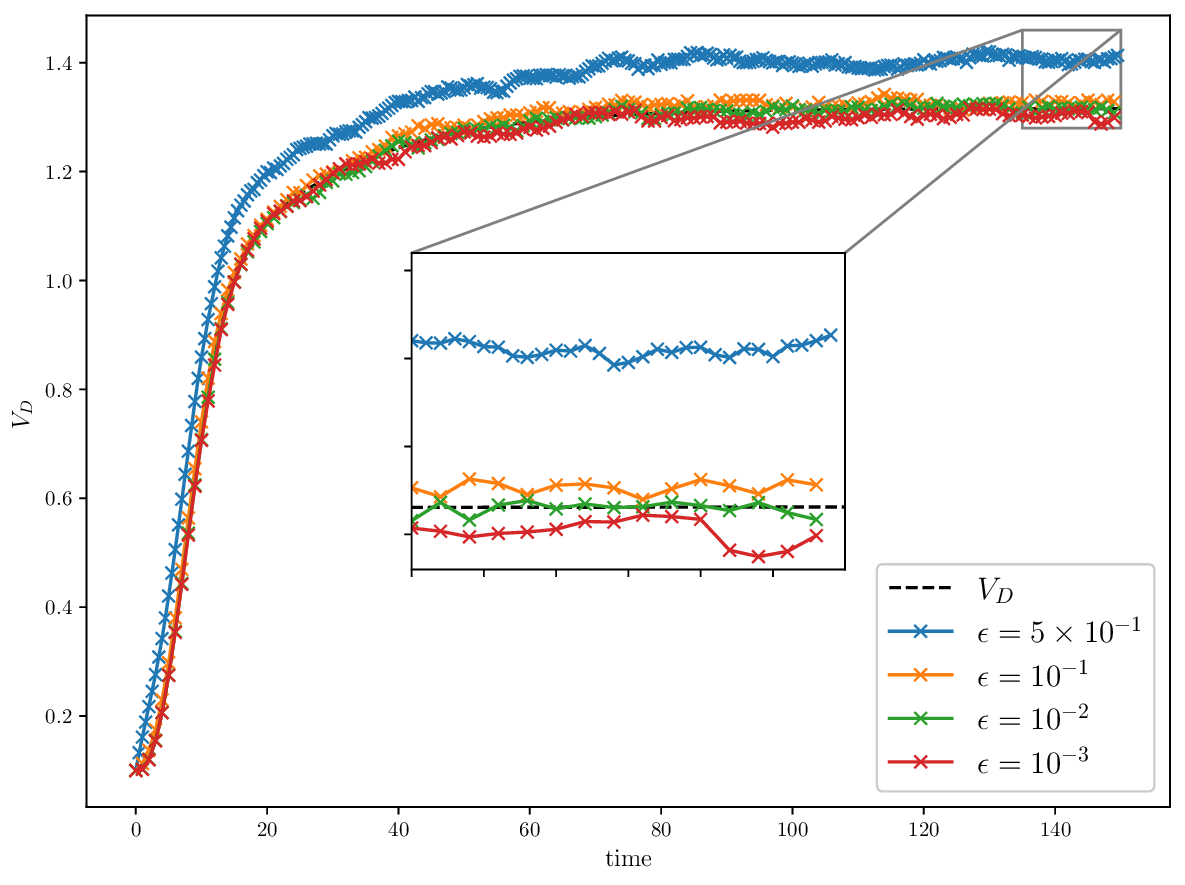}
    \end{subfigure}
    \begin{subfigure}[b]{0.48\textwidth}
        \centering
        \includegraphics[width=\textwidth]{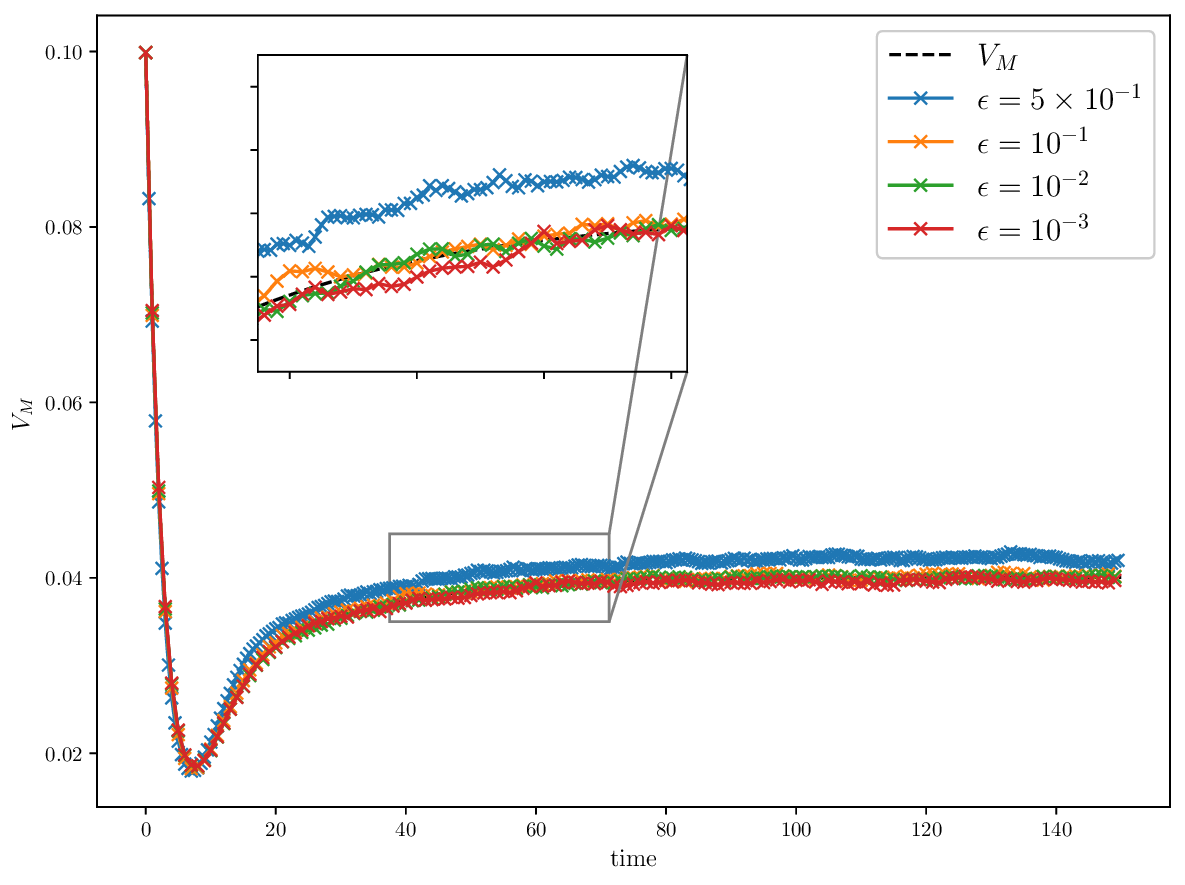}
    \end{subfigure}
    \hfill
    \begin{subfigure}[b]{0.48\textwidth}
        \centering
        \includegraphics[width=\textwidth]{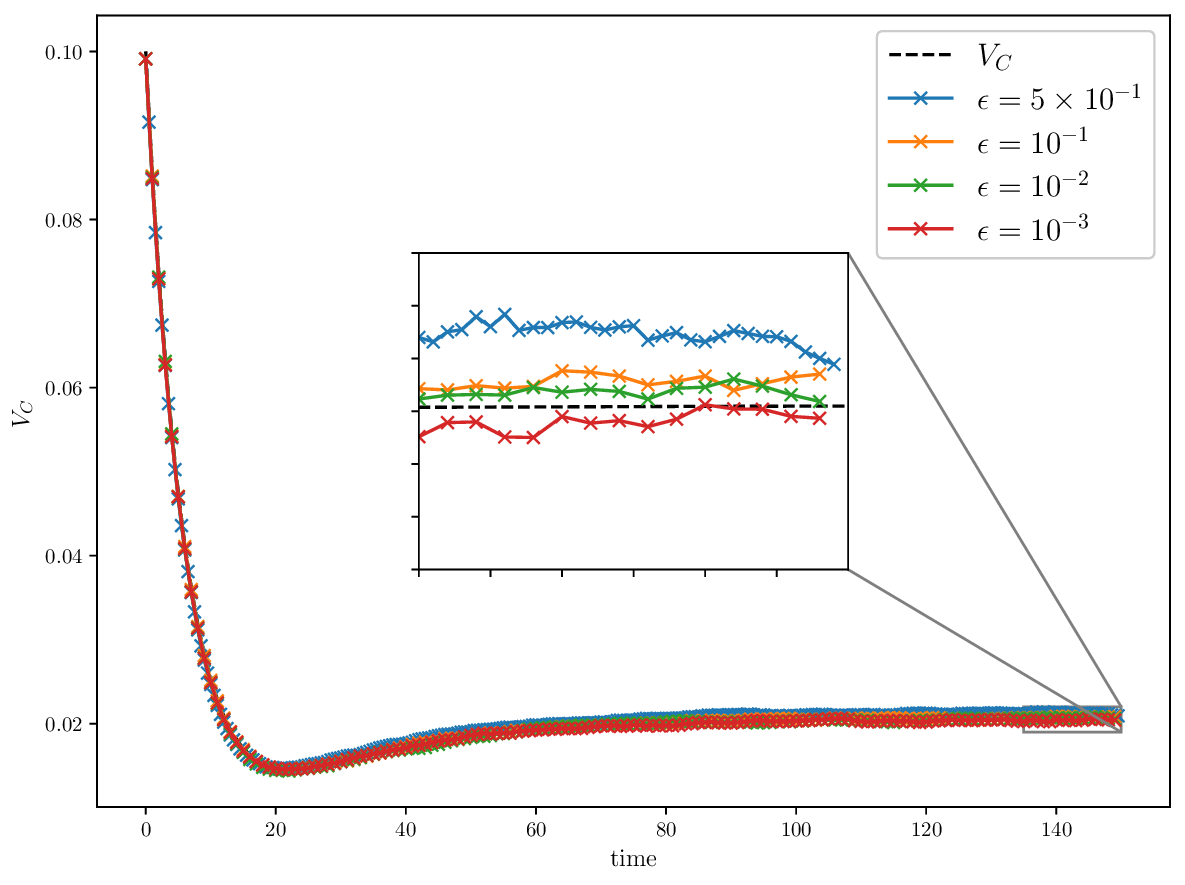}
    \end{subfigure}
    \caption{\textbf{Consistency between the kinetic model and its mean-field limit.} Solid, coloured lines highlight the dynamics of the variances, $V_J(t)$ defined via~\eqref{eq:VJ}, of the distribution functions, $f_J(x,t)$, with $J \in \{N,D,M,C\}$, of the kinetic model~\eqref{eq:system}, under the time scaling~\eqref{eq:timescalingeps} and the parameter scaling~\eqref{eq:paramscalingeps}, for the values of the scaling parameter $\epsilon$ provided in the legends. Dashed, black lines highlight the dynamics of the variances of the distribution functions of the mean-field model~\eqref{eq:MFN}-\eqref{eq:MFoBCs}. Numerical simulations were carried out under initial conditions of components defined via~\eqref{eq:init_f} and the parameter values listed in Table~\ref{Table:params}.}
    \label{fig:var_evol}
\end{figure}

\subsection{Convergence to equilibrium of the mean-field model} 
\label{sec:numres:2}
The plots in Figure~\ref{fig:sol_evol} summarise the time evolution of the components, $f_J(x,t)$ with $J \in \{N,D,M,C\}$, of the solution of the system of Fokker-Planck equations of the mean-field model and show that they converge to the equilibrium distributions, $f_J^\infty(x)$ with $J \in \{N,D,M,C\}$, defined via~\eqref{eq:finftyJ} as $t \to \infty$. 
\begin{figure}
    \centering
        \includegraphics[width=0.48\textwidth]{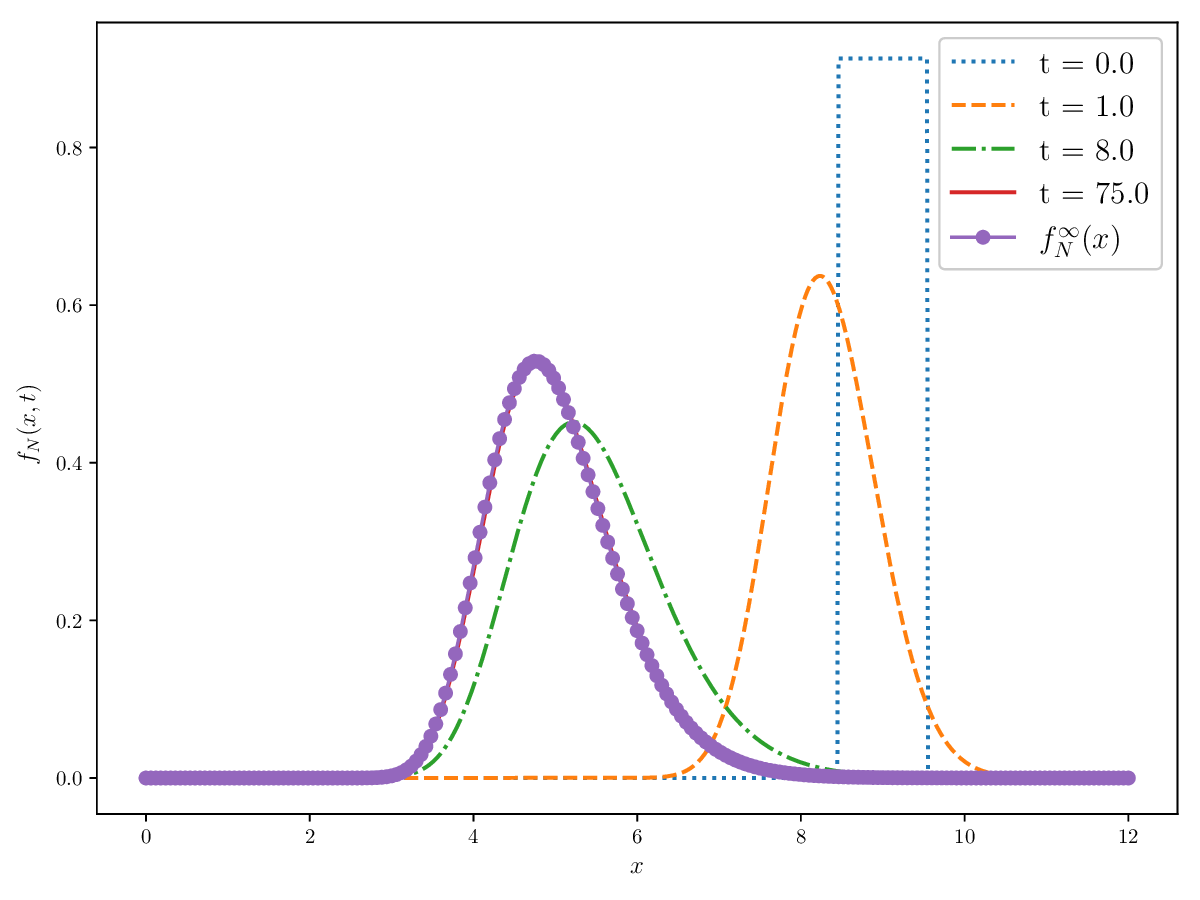}
        \includegraphics[width=0.48\textwidth]{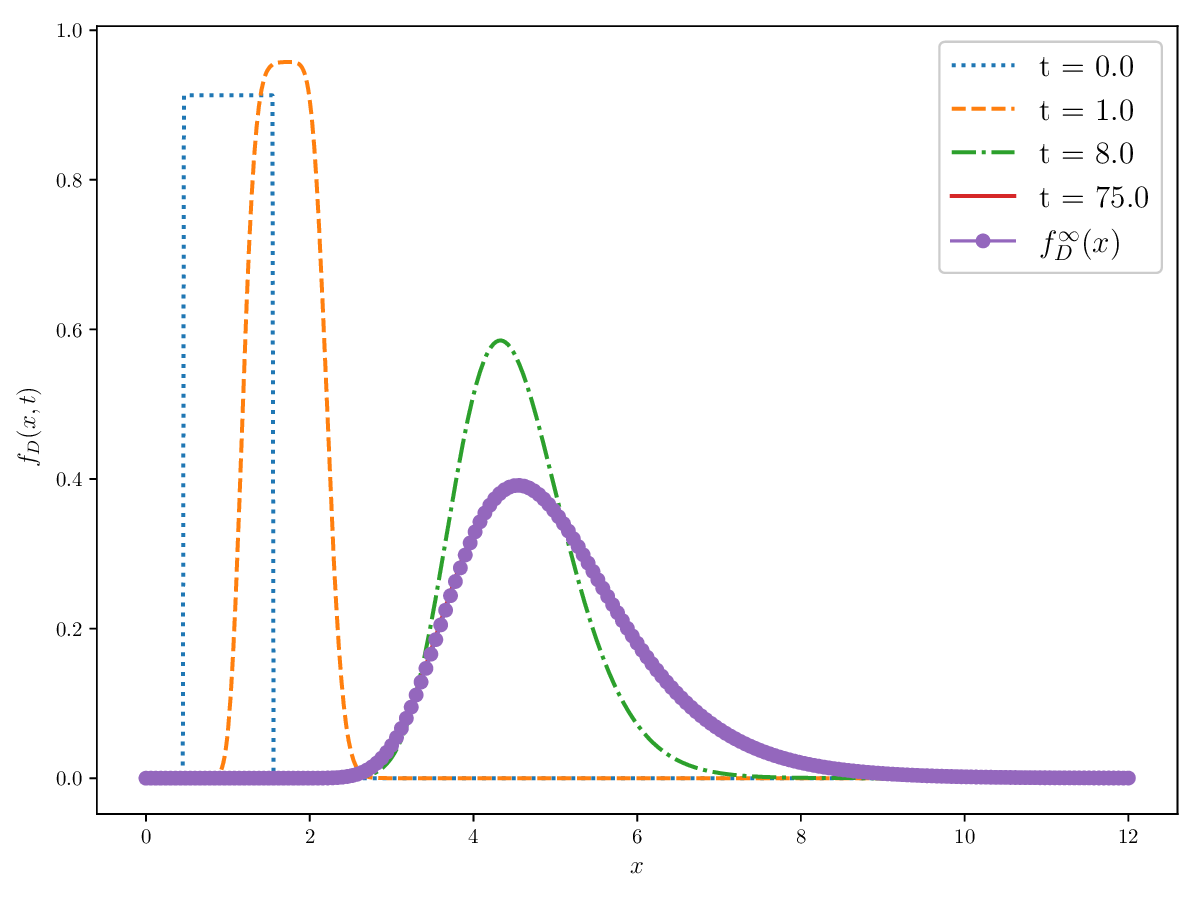}
        \includegraphics[width=0.48\textwidth]{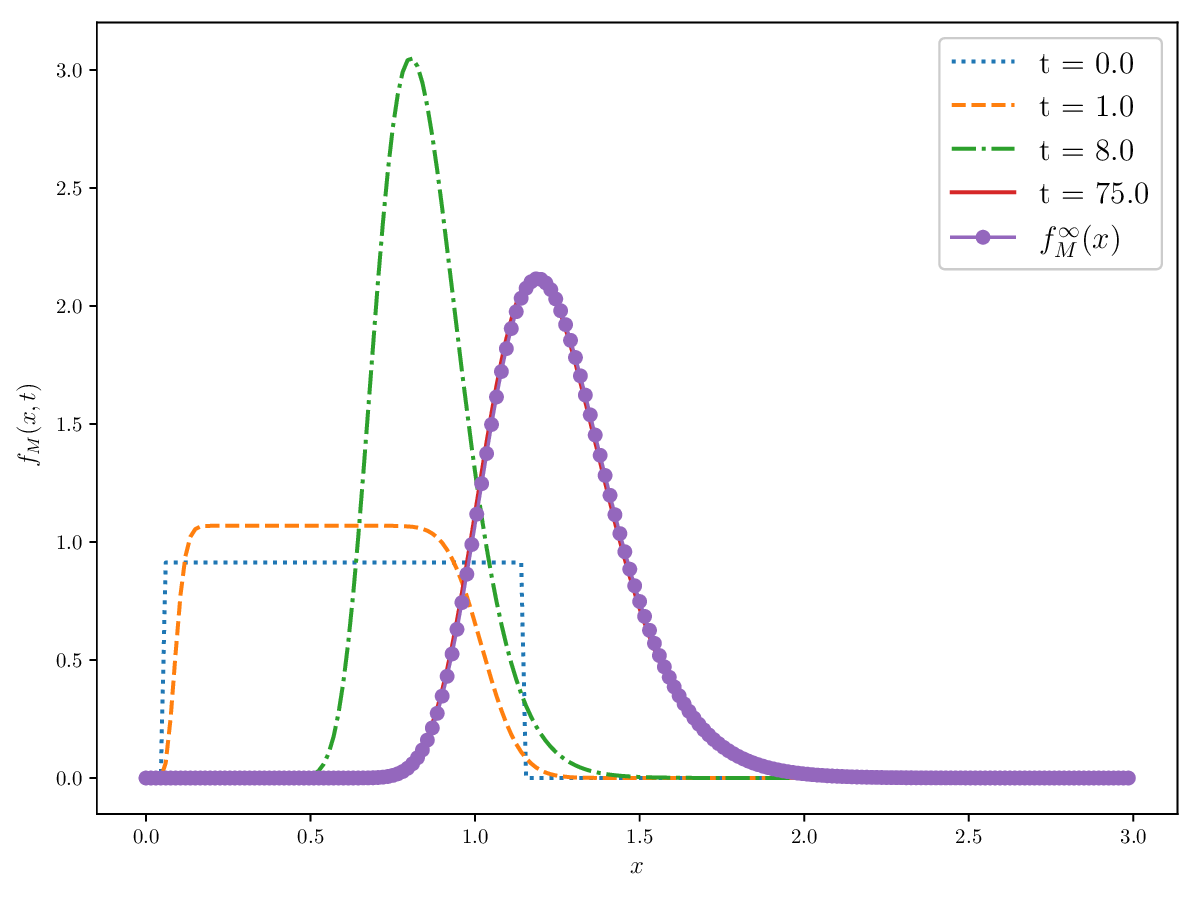}
        \includegraphics[width=0.48\textwidth]{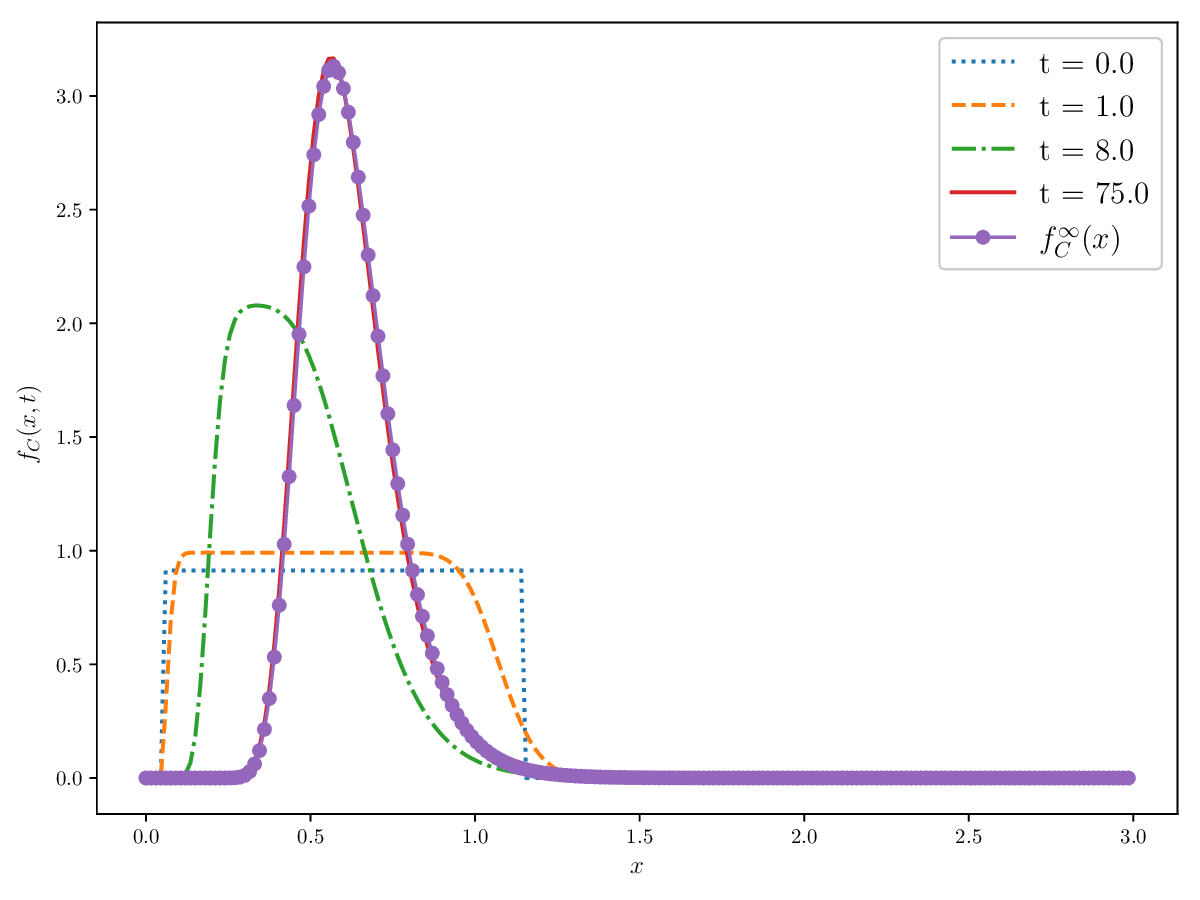}
    \caption{\textbf{Convergence to equilibrium of the mean-field model.} Plots of the distribution functions, $f_J(x,t)$, with $J \in \{N,D,M,C\}$, of the mean-field model~\eqref{eq:MFN}-\eqref{eq:MFoBCs} for $t \in \{0, 1, 8, 75\}$ (lines without markers) and the corresponding equilibrium distributions defined via~\eqref{eq:finftyJ} (lines with circle markers). Numerical simulations were carried out under initial conditions of components defined via~\eqref{eq:init_f} and the parameter values listed in Table~\ref{Table:params}.}
    \label{fig:sol_evol}
\end{figure}

To demonstrate numerically that, as per Theorem~\ref{th:convergencetoeq}, 
\[
\| f_J - f_J^\infty\|_{\dot{H}_{-p}} \to 0 \quad \text{as } t \to \infty, \qquad \frac{1}{2}<p<1, \qquad J \in \{N,D,M,C\}, 
\]
we exploit the fact that, as discussed in \cite{ABGT,MTZ}, 
\begin{equation}
\label{eq:equiv}
\| f - g\|_{\dot{H}_{-p}} = 
\dfrac{\sqrt{2}}{2p-1} \dfrac{\Gamma(\frac{3}{2}-p)}{2^{2p-1}\Gamma(p)}\mathcal E^p(f,g), 
\end{equation}
where
\begin{equation}
\begin{split}
\mathcal E^p(f,g) = &2\int_{\mathbb R_+^2}|x-y|^{2p-1}f(x)g(y)dx\,dy - \int_{\mathbb R_+^2}|x-y|^{2p-1}f(x)f(y)dx\,dy \; -  \\
& \int_{\mathbb R_+^2}|x-y|^{2p-1}g(x)g(y)dx\,dy
\end{split}
\label{def:mathcalE}
\end{equation}
is the so-called Energy Distance, see also \cite{SR}.  The plots in Figure~\ref{fig:error_grid} display the evolution of $\mathcal E^p(f_J,f_J^q)$ for different values of $p \in (\frac{1}{2},1)$, i.e. $p \in \{ \frac 5 8, \frac 3 4, \frac 7 8\}$, with $f_J(x,t)$ being the components of the numerical solution to the mean-field model~\eqref{eq:MFN}-\eqref{eq:MFoBCs} and $f_J^q(x,t)$ being computed by substituting into~\eqref{eq:fqND} and \eqref{eq:fqMC} the numerical values of $\omega_J(t)$, $J \in \{N,D,M,C\}$. These plots show that, consistently with Theorem~\ref{th:convergencetoeq}, $\mathcal E^p(f_J,f_J^q)$, and thus $\|f_J - f_J^q\|_{\dot{H}_{-p}}$, converges to zero for all $J \in \{N,D,M,C\}$ as $t \to \infty$, for all the values of $p \in (\frac{1}{2},1)$ considered. 

\begin{figure}
    \centering
        \includegraphics[width=0.48\textwidth]{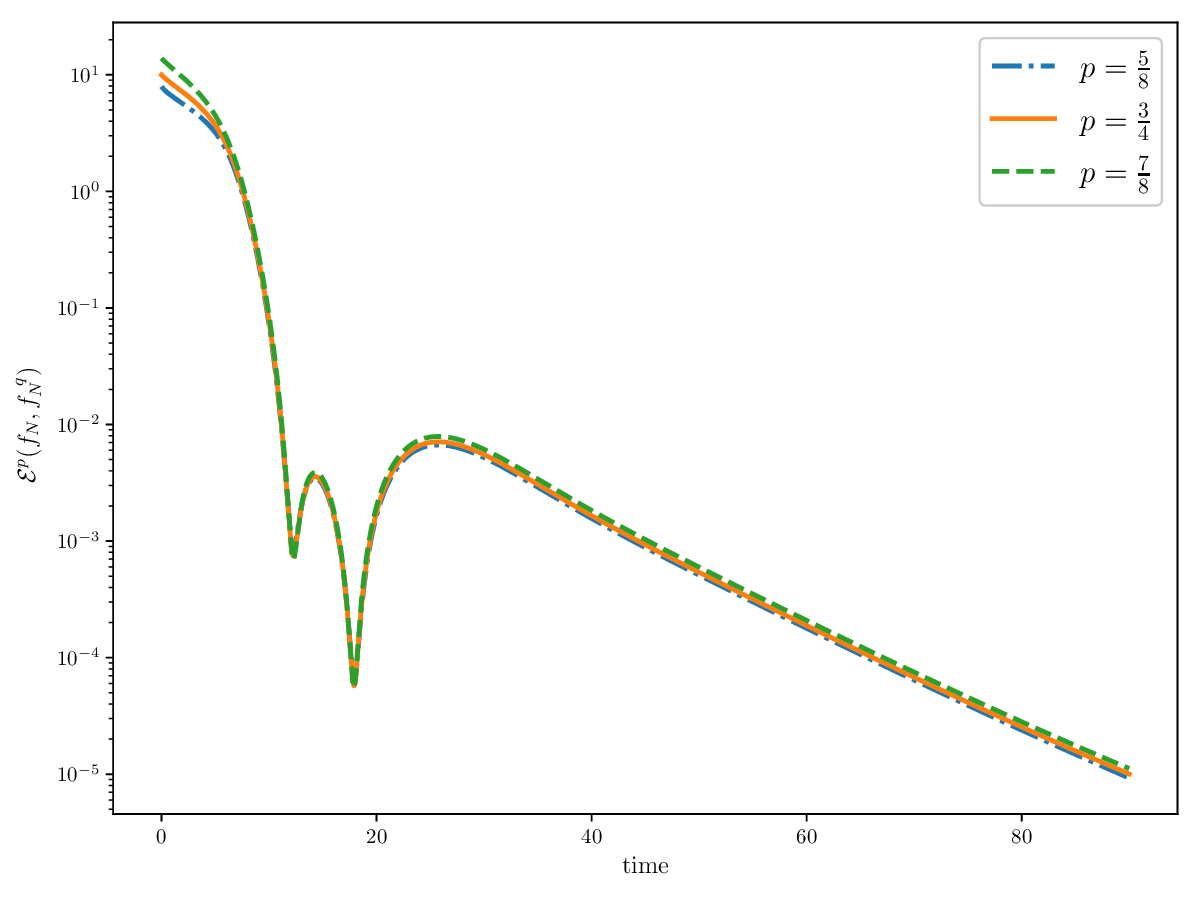}
        \includegraphics[width=0.48\textwidth]{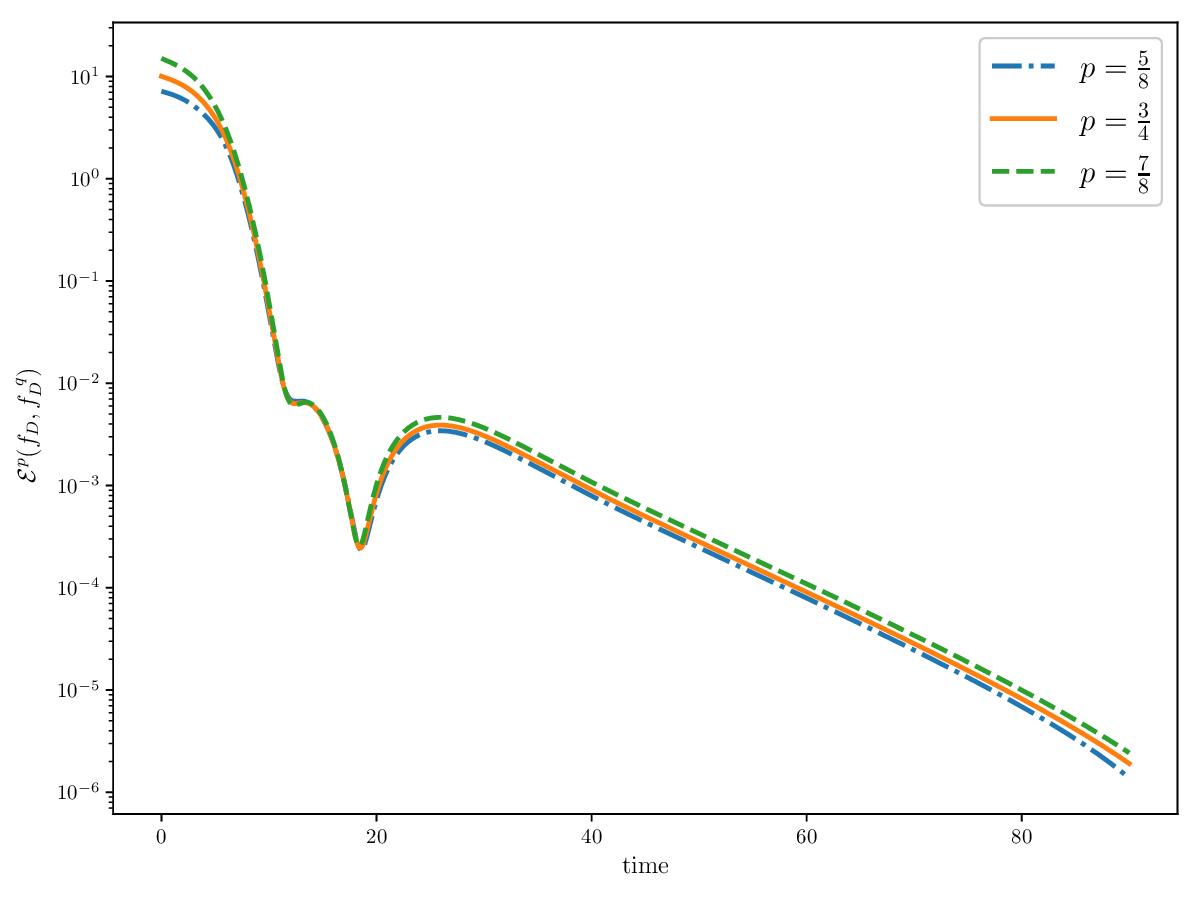}
        \includegraphics[width=0.48\textwidth]{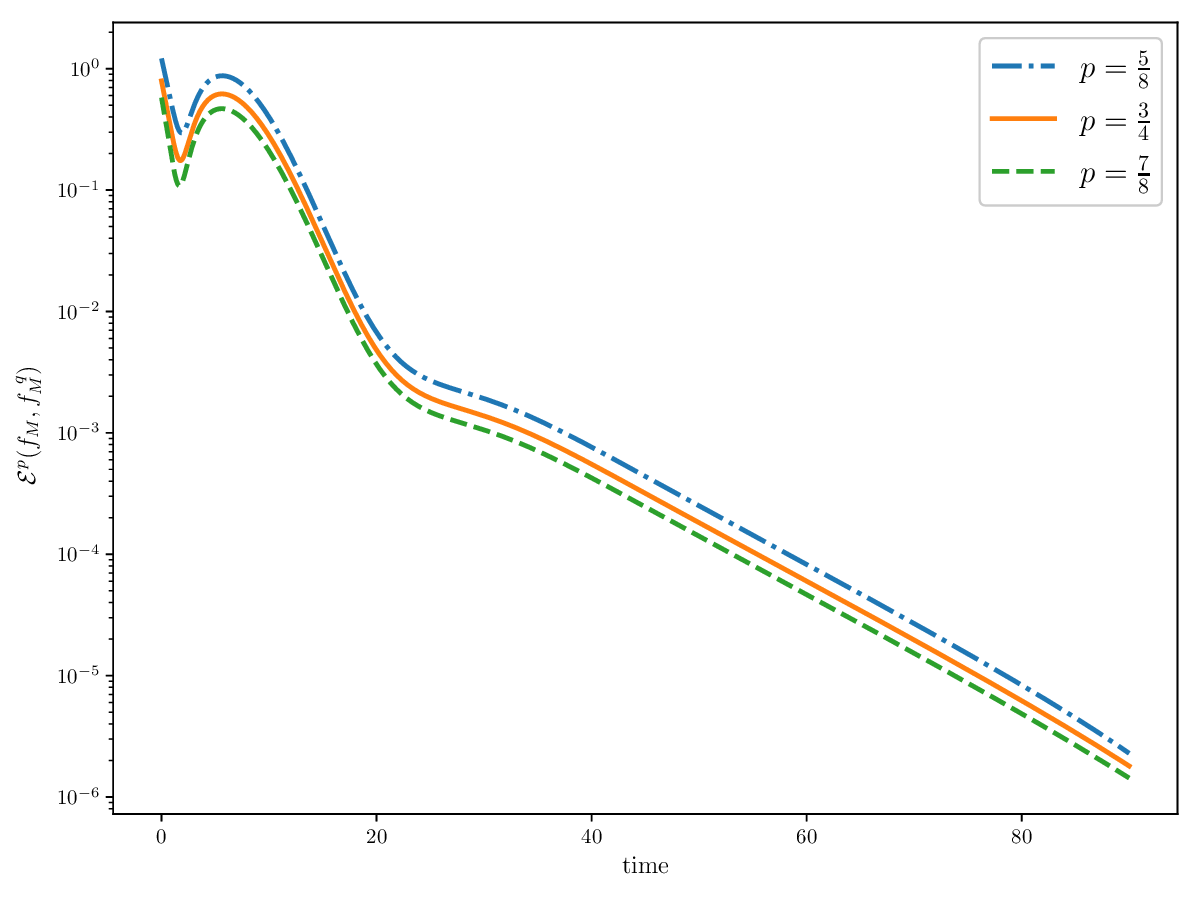}
        \includegraphics[width=0.48\textwidth]{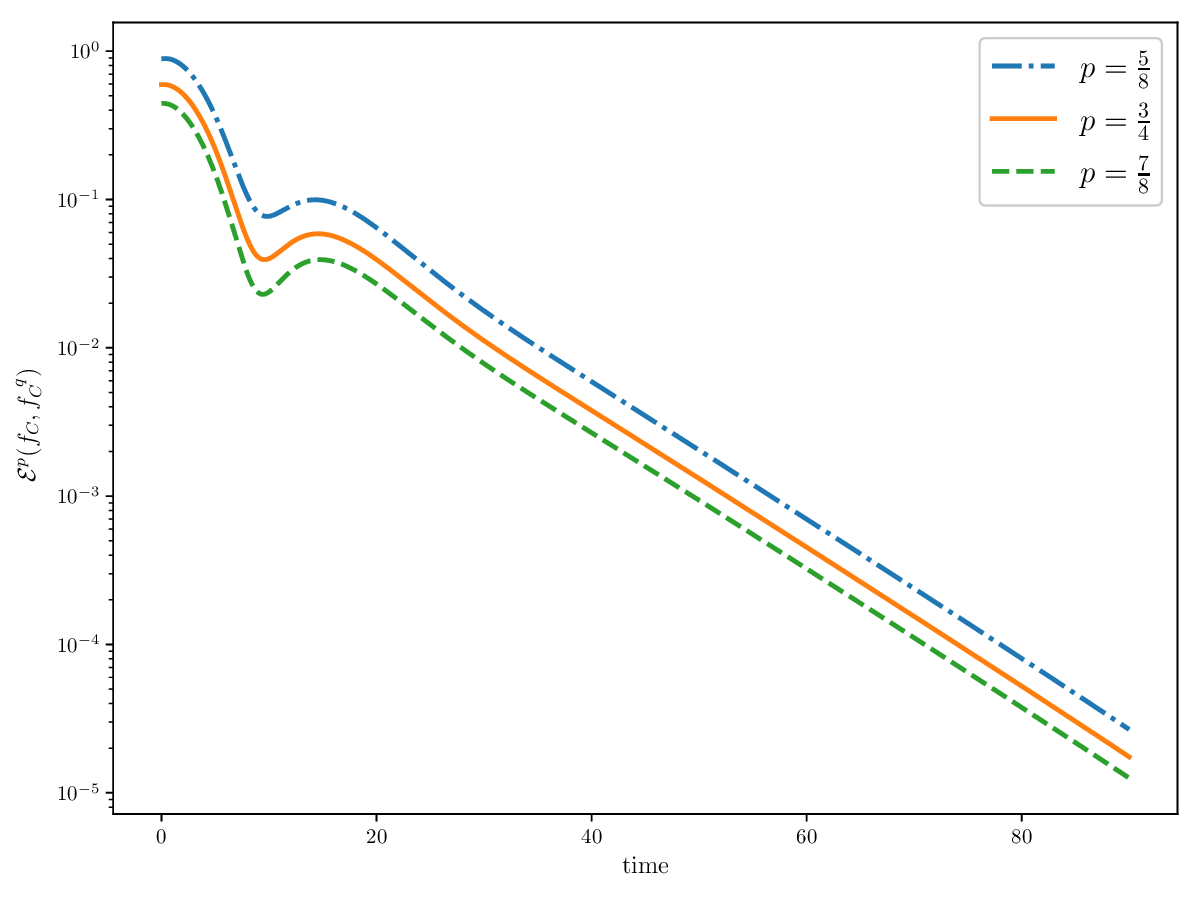}
    \caption{\textbf{Convergence to equilibrium of the mean-field model.} Plots of $\mathcal E^p(f_J,f_J^q)$ defined via~\eqref{def:mathcalE}, for $p \in (\frac{1}{2},1)$, i.e. $p \in \{ \frac 5 8, \frac 3 4, \frac 7 8\}$, with $J \in \{N,D,M,C\}$. Numerical simulations were carried out under initial conditions of components defined via~\eqref{eq:init_f} and the parameter values listed in Table~\ref{Table:params}.}
    \label{fig:error_grid}
\end{figure}

Note that the plots in Figure~\ref{fig:error_grid} indicate that $\mathcal E^p(f_J,f_J^q)$ does not decay monotonically to zero. Yet, the plots in Figure~\ref{fig:up} demonstrate that, consistently with the estimates derived in the proof of Theorem~\ref{th:convergencetoeq},
$$
\mathcal E^p(f_J,f_J^q) \leq y_J(t),  \quad \text{as } t \to \infty, \qquad J \in \{N,D,M,C\},
$$
where
\begin{equation}
y_J(t) = \phi_J(t) - 1, \qquad \phi_J(t) = \exp\left[-(3-2p)\int_0^t a_J(s)ds \right], 
\label{def:yJ}
\end{equation}
with $a_J$ being computed numerically through the formulas provided in the proof of Theorem~\ref{th:convergencetoeq}.
\begin{figure}
    \centering
        \includegraphics[width=0.48\textwidth]{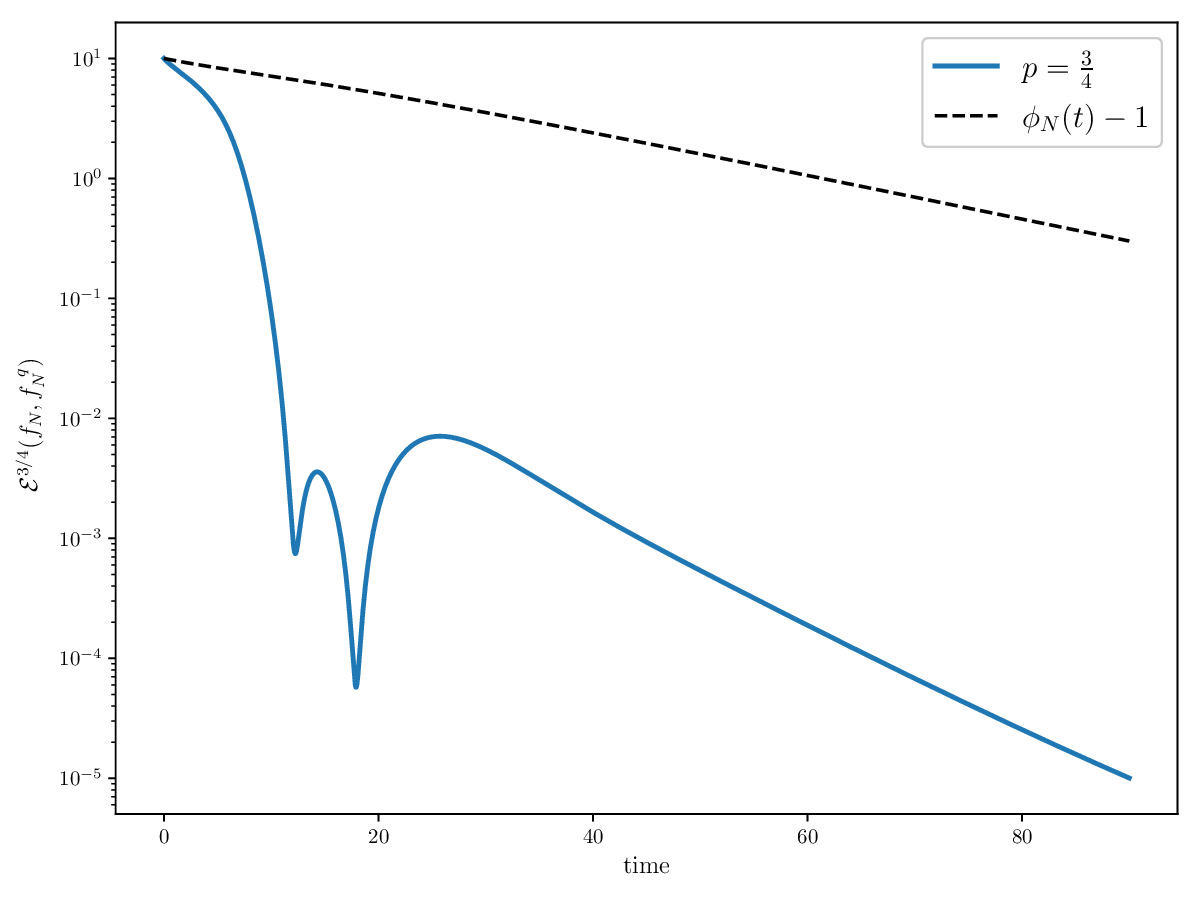}
        \includegraphics[width=0.48\textwidth]{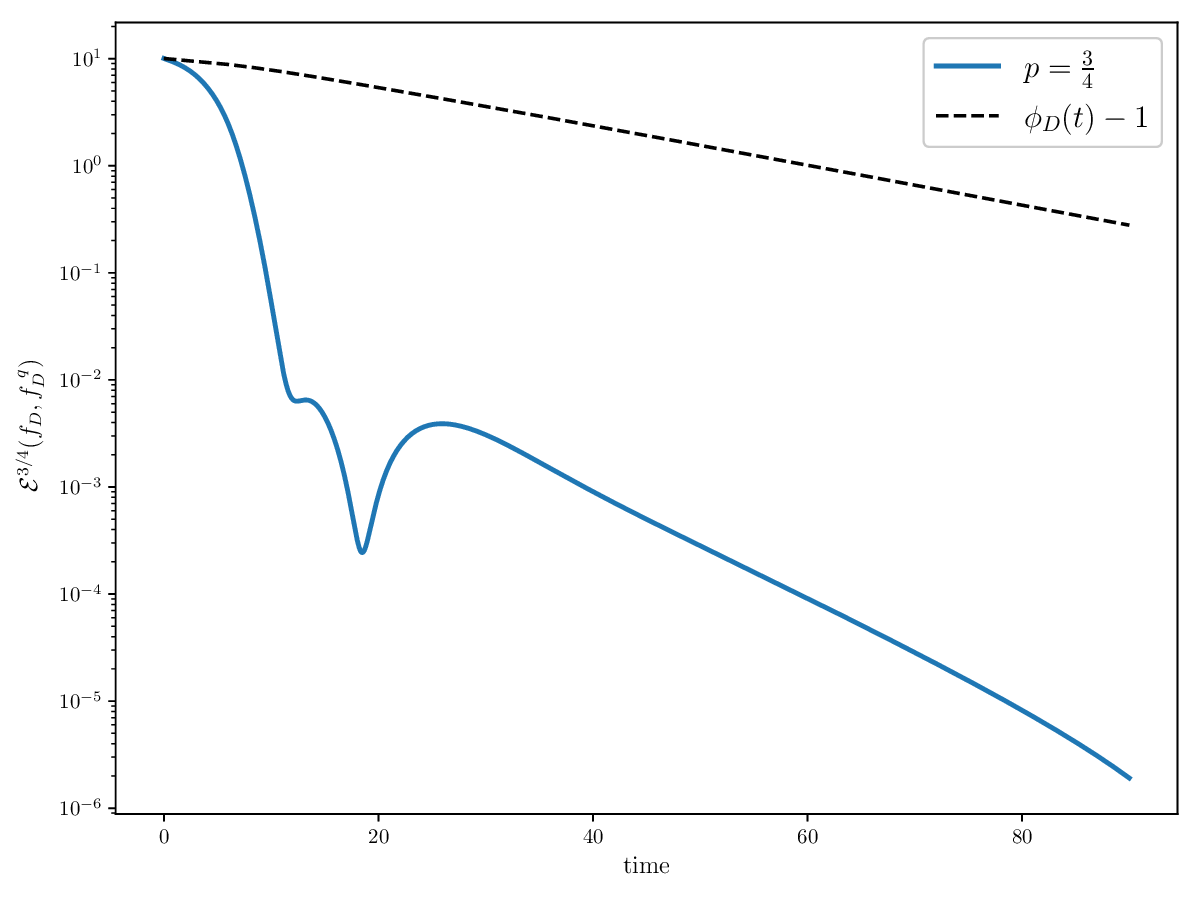}
        \includegraphics[width=0.48\textwidth]{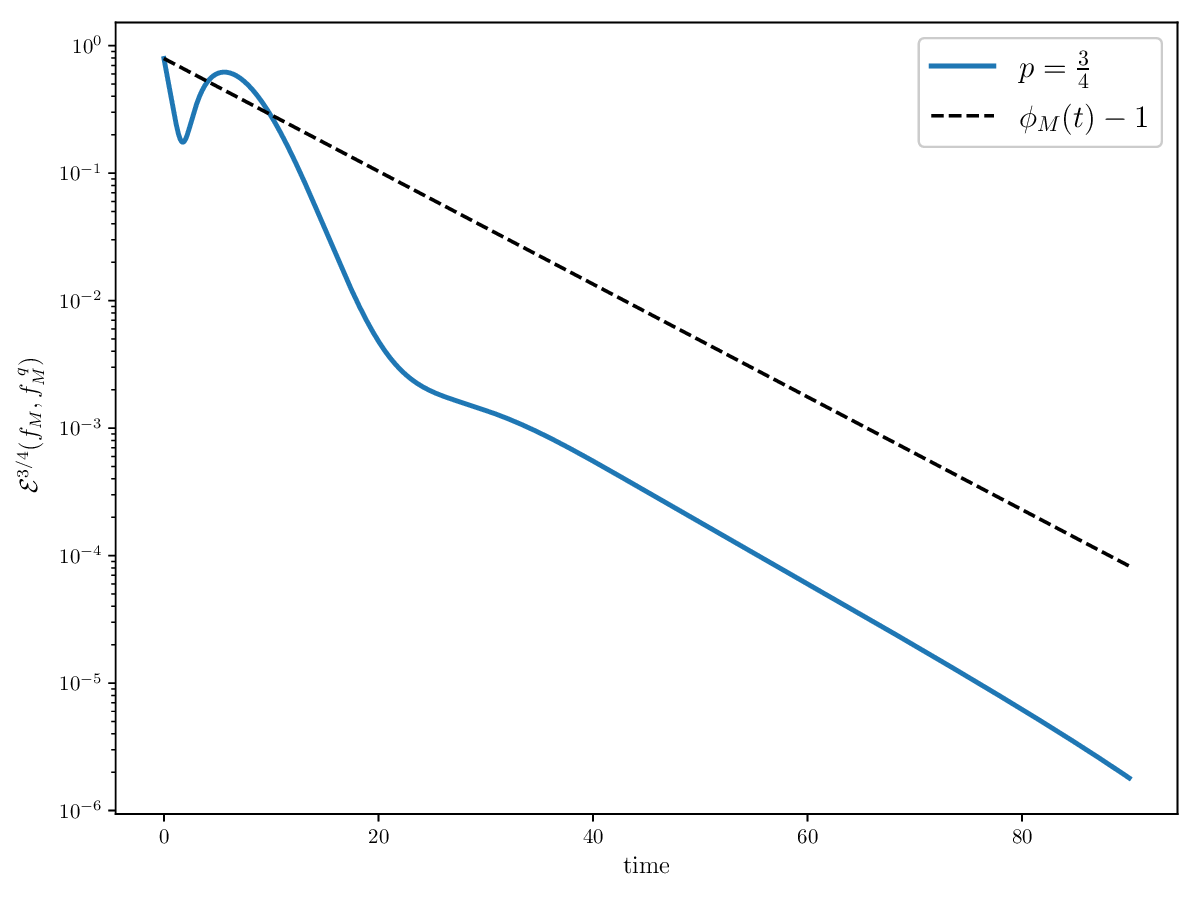}
        \includegraphics[width=0.48\textwidth]{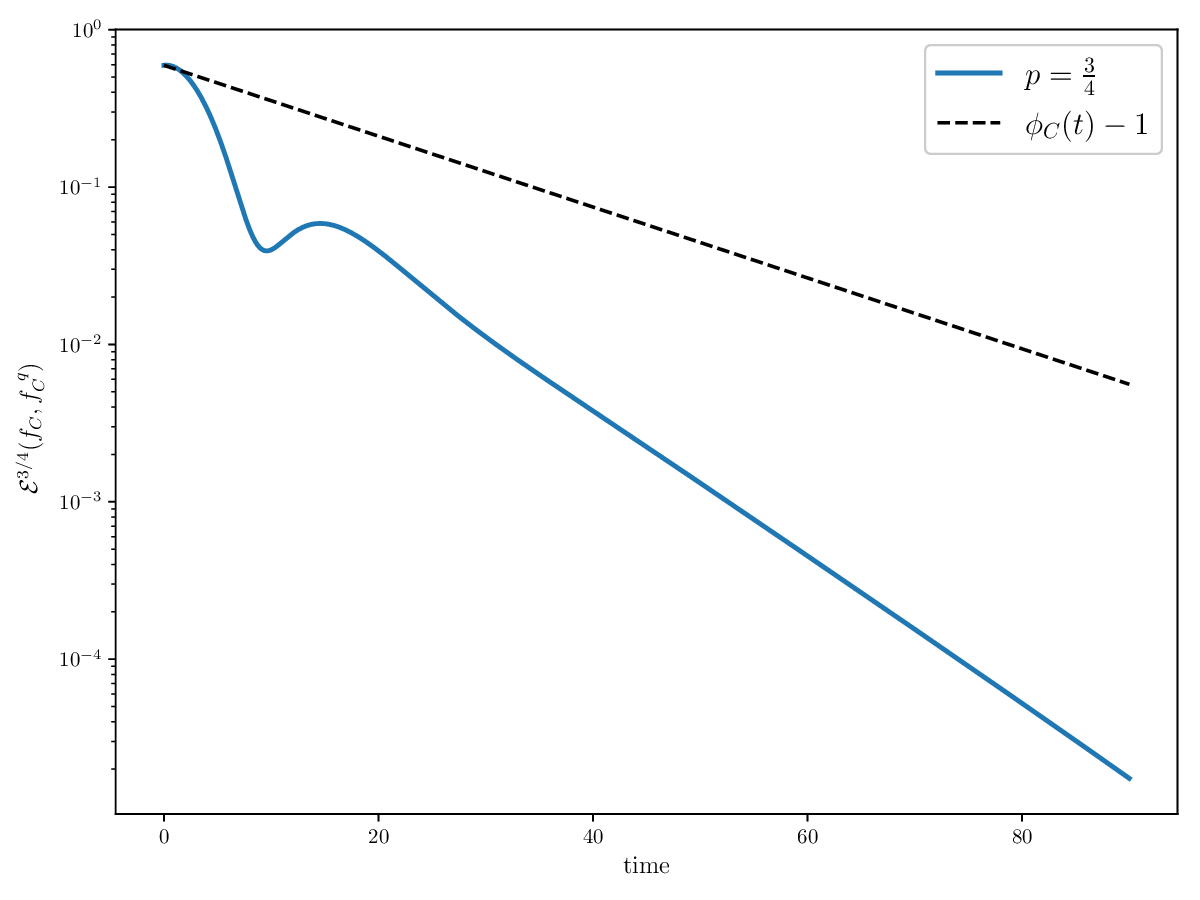}
    \caption{\textbf{Convergence to equilibrium of the mean-field model.} Plots of $\mathcal E^p(f_J,f_J^q)$ (solid lines) defined via~\eqref{def:mathcalE}, for $p = \frac{3}{4}$, and $y_J(t) = \phi_J(t) - 1$ (dashed lines) defined via~\eqref{def:yJ}, with $J \in \{N,D,M,C\}$. Numerical simulations were carried out under initial conditions of components defined via~\eqref{eq:init_f} and the parameter values listed in Table~\ref{Table:params}.}
    \label{fig:up}
\end{figure}

\newpage
\section{Conclusions and perspectives}
\label{sect:5}
In this paper, we have presented a novel kinetic model for MDs. The model comprises a system of integro-differential equations for the statistical distributions, over a large patient cohort, of the densities of muscle fibers and immune cells implicated in muscle inflammation, degeneration, and regeneration, which underpin disease development. Through an appropriate scaling procedure, we have formally derived the corresponding mean-field and macroscopic limits, thereby establishing a hierarchy of models that consistently describe the dynamics of muscle and immune cells playing a key role in the development of MDs across scales. 

Exploiting its analytical tractability, we have studied the long-time behaviour of the mean-field model. The analytical results obtained, which have been illustrated through numerical simulations, provide a full characterisation of the equilibrium cell distribution functions and establish convergence to equilibrium. These results offer a detailed mathematical depiction of how different parameters underlying cell dynamics shape the statistical distributions of muscle and immune cells over the patient cohort. As such, they also exemplify how, although the focus of this paper has been on mathematical aspects, the proposed modeling approach has the potential to serve as a quantitative tool to gain new insights into the role that different cellular processes underlying the interaction dynamics between muscle and immune cells may play in the progression of MDs.  

The present work provides a basis for future extensions, including modeling the action of therapeutic interventions. In this regard, an avenue for future research could be to investigate the effect of therapeutic approaches that keep under control the growth of the density of cytotoxic T lymphocytes induced by macrophages, in the vein of the recent experimental results on DMD presented in~\cite{WHKG}. This could be done, for instance, by modifying the definition~\eqref{eq:c} of $x_C^{\prime\prime}$ as follows 
\[
x_C^{\prime\prime} = x_C + \gamma_C x_M - \nu x_C,
\]
where the parameter $\nu \geq 0$ would be linked to the efficacy of therapy. In the mean-field limit, this would then result in the Fokker-Planck equation~\eqref{eq:MFo} for $f_C$ being replaced by the following equation for $\bar f_C$, i.e. the controlled density of  cytotoxic T lymphocytes,
$$
\dfrac{\partial}{\partial t} \bar f_C(x,t) = \dfrac{\sigma_C^2}{2} \dfrac{\partial^2}{\partial x^2} (x^2 \bar f_C(x,t) + \dfrac{\partial}{\partial x} \left\{\left[\left(\beta_C + \nu \right)x - \gamma_C \bar m_M\right] \bar f_C(x,t) \right\}.
$$
The mean density of cytotoxic T lymphocytes for $\nu>0$, i.e. the controlled mean $\bar m_C$, would then satisfy the following differential equation
\[
\dfrac{d}{dt}\bar m_C = -(\beta_C + \nu)\bar m_C + \gamma_C \bar m_M,
\]
and the equilibrium value of the mean density of normal cells, i.e. the controlled mean $\bar m^{\infty}_N$, and the corresponding variance, $\bar V_N^\infty$, would thus be
\[
\bar m_N^\infty = \dfrac{\beta_D(\beta_C + \nu) m^0}{\beta_D(\beta_C + \nu) + \beta_N \gamma_C}, \qquad \bar V_N^\infty = \dfrac{\sigma^2_N}{2\beta_N - \sigma_N^2} (\bar m_N^\infty)^2.
\]
Recalling the expression of the equilibrium value of the mean density of normal cells for $\nu=0$, i.e. the mean $m^{\infty}_N$ defined via~\eqref{eq:MNDMCinfty}, and the corresponding variance, $V^{\infty}_N$ defined via~\eqref{eq:VNDMCinfty}, one sees that
\[
\bar m_N^\infty  > m_N^\infty, \qquad \bar V_N^\infty>V_N^\infty.
\]
Hence, compared to the case without treatment, the therapy would have the beneficial effect of allowing for a larger number of normal cells to be preserved, but it would also have the detrimental effect of increasing the corresponding variance. This would imply a higher level of intra-patient heterogeneity, which might undermine therapeutic efficacy by making the design of targeted therapeutic strategies more challenging. This suggests the need for combination therapy to thwart such a detrimental effect, which points the way toward novel compelling research directions for the mathematical modeling of MDs. 

\section*{Acknowledgements}
The research underlying this paper has been undertaken within the activities of the GNFM group of INdAM (National Institute of High Mathematics). All the authors acknowledge partial support from the PRIN2022PNRR project No.P2022Z7ZAJ, European Union - NextGenerationEU. M.Z. acknowledges partial support by ICSC - Centro Nazionale di
Ricerca in High Performance Computing, Big Data and Quantum Computing, funded by European Union - NextGenerationEU.

\end{document}